\documentclass[12pt]{amsart}
\usepackage{epsf,psfrag,fullpage}
\usepackage{enumerate}
\usepackage{epsfig}
\usepackage{graphicx}
\usepackage{graphics}
\usepackage{latexsym}
\usepackage{epsfig}
\usepackage{amsfonts}
\usepackage{amsthm}
\usepackage{amsmath}
\usepackage{amssymb}
\usepackage[all]{xy}
\usepackage{amscd}
\usepackage{stmaryrd}
\usepackage{amsopn}
\usepackage[active]{srcltx}
\usepackage{mathrsfs} 
\usepackage{amsopn}
\usepackage[breaklinks=true,colorlinks=true]{hyperref}
\usepackage{xspace}

\usepackage{enumitem}

\usepackage{subfigure}

\usepackage{xcolor}

\newtheorem{theorem}{Theorem}[section]

\newtheorem{lemma}[theorem]{Lemma}

\newtheorem{cor}[theorem]{Corollary}

\newcommand{\R}{{\mathbb R}}
\newcommand{\Z}{{\mathbb Z}}

\newcommand{\T}{{\mathbb T}}

\newcommand{\op}[1]{\!\!\mathop{\rm ~#1}\nolimits}

\renewcommand{\geq}{\geqslant}
\renewcommand{\leq}{\leqslant}

\newcommand{\h}{\hbar}

\DeclareMathOperator{\Log}{Log}

\newenvironment{remark}{\refstepcounter{theorem}\par\medskip\noindent{
{\bf Remark~\thetheorem.}}}{\unskip\nobreak\hfill\hbox{$\oslash$}\par\bigskip}

\newenvironment{definition}{\refstepcounter{theorem}\par\medskip\noindent{{\bf
Definition~\thetheorem.}}}{\unskip\nobreak\hfill\hbox{}\par\medskip}

\def\scal#1#2{\left\langle #1,#2\right\rangle}
\def\Hil{\mathcal{H}\xspace}
\def\classe#1#2{\mathscr{C}^{#1}#2}

\begin{document}

\title[Asymptotics of semiclassical unitary operators]{Spectral asymptotics of semiclassical unitary operators}

\author{Yohann Le Floch\,\,\,\,\,\,\,\,\,\,{\'A}lvaro Pelayo}

\date{}



\begin{abstract}
This paper establishes an aspect of Bohr's correspondence principle, i.e. that quantum mechanics converges in the high frequency limit to classical mechanics, for commuting semiclassical unitary operators.  We prove, under minimal assumptions, that the semiclassical limit of the convex hulls of the quantum spectrum of a collection of commuting semiclassical unitary operators converges to the convex hull of the classical spectrum of the principal symbols of the operators. 
\end{abstract}

\maketitle

\section{Introduction}

One of the current leading questions in spectral theory is to what extent information about the principals symbols of an operator 
or collection of commuting operators may be detected in their joint spectrum. 

In principle, the spectrum has too little information but the surprise 
comes from the fact that \emph{sometimes} it contains all the information about the principal symbols.  

The question of determining when this is the case is spectral theoretically fundamental and fits into the recent flurry of activity 
on inverse (and direct) semiclassical spectral problems \cite{Zel,SanInverse,ColGui,Col,ChaPelVu,PelVuFF,LFPelVu,MR3129048, MR2801777}. It is
originally motivated by Bohr's correspondence principle: that quantum mechanics converges in the high frequency (i.e. semiclassical) limit to
classical mechanics. This principle can have many interesting manifestations.

Four years ago, Pelayo, Polterovich and V{\~u} Ng{\d{o}}c proposed in \cite{PelPolVu} a minimal set of axioms that a collection of commuting 
semiclassical self-adjoint operators should satisfy in order for the convex hull of the semiclassical limit of their joint spectrum to converge to the
convex hull of the joint image of the principal symbols (a subset of Euclidean space). 

The result by these authors is not known to hold, however, for other types of operators which are also very important in analysis and physics, such as unitary operators. 
These are of special interest in symplectic geometry in view of the recent breakthrough by Susan Tolman~\cite{MR3735631} who has shown that there are many symplectic
non-Hamiltonian actions with finitely many fixed points (on compact manifolds). 

All such actions admit a $\mathbb{S}^1$\--valued momentum map and their quantization
is a semiclassical unitary operator. This is the original motivation of our work below. See also for further motivations \cite[Section 5.2]{PelBAMS17} and \cite{MR3506787}. Here we do not
use Fourier integral operators, which are used to quantize symplectomorphisms, but instead we are concerned with 
genuine pseudodifferential or Berezin-Toeplitz operators quantizing $\mathbb{S}^1$-valued functions on the phase space.

The goal of this paper is to precisely address the deficiency in the literature by studying an important property of the 
semiclassical asymptotics of unitary operators: the relation between the joint spectrum of the operators and the joint image of the principal symbols of the operators.

In order to do this precisely, and with the maximum generality, we start by defining the notion of semiclassical quantization of a manifold.

Let $M$ be a connected manifold. Let $\mathcal{A}_0$ be a subalgebra of $\classe{\infty}{(M,\mathbb{C})}$ containing the constants and the compactly supported functions, and stable by complex conjugation. Assume also that if $f \in \mathcal{A}_0$ never vanishes, then $1/f$ also belongs to $\mathcal{A}_0$; {finally, assume that whenever $f \in \mathcal{A}_0$ is real-valued and bounded and $g \in \classe{\infty}{(\mathbb{R},\mathbb{R})}$ is bounded, $g \circ f$ also belongs to $\mathcal{A}_0$}. 

\begin{definition}\label{0} Suppose that $I \subset (0,1]$ accumulates at zero. 
A \emph{semiclassical quantization} of $(M,\mathcal{A}_0)$ is a family of complex Hilbert spaces $(\Hil_{\hbar})_{\hbar \in I}$ together with a family of $\mathbb{C}$-linear maps $\mathrm{Op}_{\hbar}: \mathcal{A}_0 \to \mathcal{L}(\Hil_{\hbar})$ satisfying for all $f,g \in \mathcal{A}_0$:
\begin{enumerate}[label=(Q\arabic*)]
\item \label{item:composition} if $f$ and $g$ are bounded, then the composition $\mathrm{Op}_{\hbar}\left(f \right) \mathrm{Op}_{\hbar}\left(g\right)$ is well-defined and
$$\| \mathrm{Op}_{\hbar}(f) \mathrm{Op}_{\hbar}(g) - \mathrm{Op}_{\hbar}(fg) \| = \mathcal{O}(\hbar).$$ 
({\em composition});
\item \label{item:reality} for every $\hbar \in I$, $\mathrm{Op}_{\hbar}(f)^* = \mathrm{Op}_{\hbar}\left(\bar{f} \ \right)$.
 ({\em reality});
\item \label{item:normalization} $ \mathrm{Op}_{\hbar}(1) = \mathrm{Id}$. 
({\em normalization}).
\item \label{item:garding}if $f \geq 0$, then there exists  $C > 0$ such that for every $\h\in I$, $\mathrm{Op}_{\hbar}(f) \geq -C\hbar \ \mathrm{Id}$.
(\emph{quasi-positivity});
\item \label{item:non-degeneracy} if $f \neq 0$ has compact support, then $\mathrm{Op}_{\hbar}(f)$ is bounded for every $\hbar \in I$ and 
$$ \liminf_{\hbar \to 0} \| \mathrm{Op}_{\hbar}(f)\| > 0.$$ 
({\em non-degeneracy}).
\item \label{item:product} if $g$ has compact support, then for every $f \in \mathcal{A}_0$, $\mathrm{Op}_{\hbar}(f) \mathrm{Op}_{\hbar}(g)$ is bounded and 
$\| \mathrm{Op}_{\hbar}(f) \mathrm{Op}_{\hbar}(g) - \mathrm{Op}_{\hbar}(fg) \| = \mathcal{O}(\hbar).$
({\em product formula}).
\end{enumerate}
If such a semiclassical quantization exists, we say that $M$ is \emph{quantizable}. 
\end{definition}

Suppose that  $M$  is quantizable. Then there is a way to associate, to each $f \in \classe{\infty}{(M,\mathbb{C})}$, a family $(\mathrm{Op}_{\hbar}(f))_{\hbar \in I}$ of operators acting on Hilbert spaces $(\Hil_{\hbar})_{\hbar \in I}$, and such that this construction respects certain axioms (Definition \ref{so}). When this is the case, one can work in the reverse direction and associate to such a family $(T(\hbar))_{\hbar \in I}$ a function $f \in \classe{\infty}{(M,\mathbb{C})}$ such that $T(\hbar) = \mathrm{Op}_{\hbar}(f) + \mathcal{O}(\hbar)$, which is called
 \emph{the principal symbol of $T(\hbar)$}. 

If  $(U_1(\hbar),\dots, U_d(\hbar))$ are  pairwise commuting unitary 
semiclassical operators on $M$ then necessarily their principal symbols are $\mathbb{S}^1$\--valued\footnote{this is proved later in Lemma \ref{lm:symb_unitary}.}.
Hence the image of a collection of such operators is a subset of the $d$\--torus.  For a subset $A$ of $\T^d$  let $\textup{Convex Hull}_{\T^d}(A)$ be its convex hull in the torus.   

The notion of convex hull in tori is not obvious (more precisely, one cannot take the naive notion), see the second appendix.
  The details of the second appendix are plentiful but fortunately they are 
  not needed to understand the statement above, only the definition of convex hull is needed, but they are needed for the proofs. 
  
    We can now
  state the main result of the paper, which studies convergence in the Hausdorff 
  distance\footnote{The \emph{Hausdorff distance} (see e.g.\cite[Definition 7.3.1]{Bur}) between two subsets $A\subset X$ and $B\subset X$ of a metric space $(X,d)$ is the quantity
${\rm d}_H^X(A,\,B):= \inf\left\{\varepsilon > 0\,\, | \,\ A \subseteq B_\varepsilon \ \mbox{and}\ B \subseteq A_\varepsilon\right\}$, 
where for any subset $S$ of $X$, and for any $\epsilon>0$, the set $S_{\varepsilon}$ is defined as $S_\varepsilon := \cup_{s \in S} \{x \in X \, \, | \,\, d(s, x) \leq \varepsilon\}$. Recall that if ${\rm d}_H^X(A,\,B)=0$ and $A,B$ are closed sets, then $A=B$. When $X = \R^d$ with its Euclidean norm, we will simply use the notation $d_{H}$ for the Hausdorff distance. }
 for a large class of unitary operators:

\begin{theorem} \label{theo:toeplitz}
Let $M$ be a quantizable manifold and let $U_1(\hbar),\dots, U_d(\hbar)$ be semiclassical commuting unitary operators on $M$.
Let $f_j$ be the principal symbol of $U_j(\hbar)$ and assume that no $f_j$ is onto,  and all images  $f_j(M)$'s and the joint image 
$(f_1,\ldots, f_d)(M)$ are closed. 
 Then from $\{\textup{Convex Hull}_{\T^d}(\textup{JointSpec}(U_1(\hbar),\dots, U_d(\hbar))\}_{\hbar \in I}$, one can recover the convex hull of  $(f_1,\ldots,f_f)(M)$.
 Furthermore, under generic assumptions, the semiclassical limit of this family as $\hbar \to 0$ exists and equals ${\rm Convex\,\, Hull}_{\T^d}\,{(f_1,\ldots, f_d)(M)}$.
 \end{theorem}

A version with the generic assumptions is Theorem~\ref{thm:CH_simpleF}. This result  has an interesting application to symplectic geometry of group actions (Section \ref{sect:circle_BTO}). 

\subsection*{Structure of the paper} In Section~\ref{sec:sc} we review basic operator theory and discuss the implications  derived from having
a semiclassical quantization, which are used in the proof of the theorem above. In Section~\ref{sec:proofofmaintheorem} we prove 
the theorem; the proof is long and for clarity we divide it into several subsections. In Section~\ref{sect:circle_BTO} we explain in detail the application to symplectic geometry above. We conclude with a section giving a few remarks and two appendices: one on semiclassical operators  
and one on convex hull in tori. The readers may consult the appendices as needed in order to understand the proof of the main theorem.

\section{Semiclassical quantization} \label{sec:sc}

The goal of this section is to prepare the grounds for the proof of the main theorem, by identifying important consequences of the axioms
in the notion of quantization  (Definition~\ref{0}). We will also explain the precise meaning of semiclassical operator, and principal symbol, 
in this context.

\begin{cor}
If $f \in \mathcal{A}_0$ is bounded, then the operator $\mathrm{Op}_{\hbar}(f)$ is bounded and satisfies
\label{cor:bounded}
\begin{equation} \|\mathrm{Op}_{\hbar}(f)\| \leq \|f \|_{\infty} + \mathcal{O}(\hbar), \label{eq:bound}\end{equation}
where $\|f\|_{\infty}$ is the uniform norm of $f$.
\end{cor}

\begin{proof}
Axioms \ref{item:normalization} and \ref{item:garding} yield that 
$\mathrm{Op}_{\hbar}\left(|f|^2 \right)  \leq  \|f \|_{\infty}^2 \ \mathrm{Id} + \mathcal{O}(\hbar).$
Since $\mathrm{Op}_{\hbar}(|f|^2)$ is self-adjoint, this implies, by formula (\ref{eq:rayleigh}), that its norm satisfies $\left\| \mathrm{Op}_{\hbar}(|f|^2) \right\| \leq \|f \|_{\infty}^2 +  \mathcal{O}(\hbar)$; because of axioms \ref{item:composition} and \ref{item:reality}, this means that
\[ \left\| \mathrm{Op}_{\hbar}(f)^* \mathrm{Op}_{\hbar}(f) \right\| \leq \|f \|_{\infty}^2 +  \mathcal{O}(\hbar).\]
But this, in turn, yields the boundedness of $\mathrm{Op}_{\hbar}(f)$; indeed, if $u \in \Hil$ belongs to the domain of $\mathrm{Op}_{\hbar}(f)$, then we get by the Cauchy-Schwarz inequality that
\[ |\scal{\mathrm{Op}_{\hbar}(f)u}{\mathrm{Op}_{\hbar}(f)u}| = |\scal{\mathrm{Op}_{\hbar}(f)^* \mathrm{Op}_{\hbar}(f)u}{u}|  \leq \left\| \mathrm{Op}_{\hbar}(f)^* \mathrm{Op}_{\hbar}(f) u \right\| \| u \|.\]
Therefore, we obtain that $\|\mathrm{Op}_{\hbar}(f) u\| \leq \sqrt{\left\| \mathrm{Op}_{\hbar}(f)^* \mathrm{Op}_{\hbar}(f) \right\|} \ \| u \|,$
which implies that $\mathrm{Op}_{\hbar}(f)$ is bounded and that its norm satisfies  (\ref{eq:bound}).
\end{proof}

We state another useful corollary of our axioms regarding  invertibility.

\begin{cor}
\label{cor:inverse}
Let $f \in \mathcal{A}_0$ be bounded. The following  are equivalent:
\begin{itemize}
\item there exists $\hbar_0 \in I$ such that for every $\hbar \leq \hbar_0$, $\mathrm{Op}_{\hbar}(f)$ is invertible and the norm of its inverse is uniformly bounded in $\hbar$, 
\item there exists $c > 0$ such that $|f| \geq c$.
\end{itemize}
\end{cor}

\begin{proof}
Note that since $f$ is bounded, the previous corollary yields that $\mathrm{Op}_{\hbar}(f)$ is bounded with norm smaller than 
$ \|f\|_{\infty} + \mathcal{O}(\hbar).$
Assume that $\mathrm{Op}_{\hbar}(f)$ is invertible for $\hbar \leq \hbar_0$ with 
$ \| \mathrm{Op}_{\hbar}(f)^{-1} \| \leq 1/c$
for every $\hbar \leq \hbar_0$, for some constant $c > 0$. Then from
$\mathrm{Op}_{\hbar}(f)^{-1} \mathrm{Op}_{\hbar}(f) = \mathrm{Id},$
we derive the following inequality:
\begin{equation} \forall u \in \Hil_{\hbar} \qquad  \| \mathrm{Op}_{\hbar}(f) u \| \geq \frac{\| u \|}{\| \mathrm{Op}_{\hbar}(f)^{-1} \|} \geq c \|u\|. \label{eq:bound_below}\end{equation}
Let $m \in M$ and choose a compact set $\widetilde{K} \subset M$ such that $|f(p) - f(m)| \leq \tfrac{c}{4}$ for all $p \in \widetilde{K}$. Let $\chi \geq 0$ be a smooth function identically equal to one on a compact set $K$ containing $m$ included in the interior of $\widetilde{K}$ and with compact support contained in $\widetilde{K}$. We claim that there exists $u_{\hbar} \in \Hil_{\hbar}$ of unit norm such that
\begin{equation} u_{\hbar} =  \mathrm{Op}_{\hbar}(\chi) u_{\hbar} + \mathcal{O}(\hbar).  \label{eq:vector}\end{equation}
This claim is established in Step 3 of the proof of Lemma 11 in \cite{PelPolVu}, but we present a sketch of its proof for the sake of completeness. Let $\eta$ be a smooth, not identically vanishing function supported on $K$. By axiom \ref{item:non-degeneracy}, there exists $\gamma > 0$ such that 
$\| \mathrm{Op}_{\hbar}(\eta)\| \geq \gamma$
for every $\hbar \leq \hbar_0$, so there exists some $v_{\hbar} \in \Hil_{\hbar}$ of norm 1 and such that 
$\| \mathrm{Op}_{\hbar}(\eta) v_{\hbar} \| > \gamma/2. $
Choose $u_{\hbar}$ as follows:
\[ u_{\hbar} = \frac{1}{\| \mathrm{Op}_{\hbar}(\eta) v_{\hbar} \|}  \mathrm{Op}_{\hbar}(\eta) v_{\hbar}.\]
Thanks to axiom \ref{item:product}, we obtain 
\[ \mathrm{Op}_{\hbar}(\chi) u_{\hbar} =  \frac{1}{\| \mathrm{Op}_{\hbar}(\eta) v_{\hbar} \|}  \mathrm{Op}_{\hbar}(\chi \eta) v_{\hbar} + \mathcal{O}(\hbar) \]
which allows us to conclude that $u_{\hbar}$ satisfies formula (\ref{eq:vector}), since $\chi \eta = \eta$.

We choose such a $u_{\hbar}$. By axiom \ref{item:product}, we get that
\[ \| \mathrm{Op}_{\hbar}(\chi f) u_{\hbar} - \mathrm{Op}_{\hbar}(f)\mathrm{Op}_{\hbar}(\chi) u_{\hbar}  \| = \mathcal{O}(\hbar). \] 
Combining this estimate with the fact that $u_{\hbar}$ satisfies equation (\ref{eq:vector}) yields
$\| \mathrm{Op}_{\hbar}(\chi f) u_{\hbar} - \mathrm{Op}_{\hbar}(f) u_{\hbar}  \| = \mathcal{O}(\hbar)$
and using equations (\ref{eq:bound}) and (\ref{eq:bound_below}), this gives 
\[ \|\chi f \|_{\infty} + \mathcal{O}(\hbar) \geq \| \mathrm{Op}_{\hbar}(\chi f) u_{\hbar}  \| \geq \| \mathrm{Op}_{\hbar}(f) u_{\hbar}\| + \mathcal{O}(\hbar)  \geq c + \mathcal{O}(\hbar). \]
By choosing $\hbar$ sufficiently small, this yields $\|\chi f \|_{\infty} \geq c/2$. Since $0 \leq \chi \leq 1$, this means that there exists $p \in \widetilde{K}$ such that $|f(p)| \geq c/2$. But by our choice of $\widetilde{K}$, this yields $|f(m)| \geq |f(p)| - \tfrac{c}{4} \geq \tfrac{c}{4}$.\\

Conversely, assume that $|f| \geq c$ for some constant $c > 0$. Then $1/f$ is bounded, thus axiom \ref{item:composition} implies that 
$\mathrm{Op}_{\hbar}(f)\mathrm{Op}_{\hbar}\left(\frac{1}{f} \right) = \mathrm{Id} + R_{\hbar}$
where $R_{\hbar}$ is bounded with norm $\mathcal{O}(\hbar)$. By a standard result (see for instance \cite[Theorem A3.30]{HislopSigal}), there exists $\hbar_1 \in I$ such that $\mathrm{Id} + R_{\hbar}$ is invertible whenever $\hbar \leq \hbar_{1}$, thus for such $\hbar$
\[ \mathrm{Op}_{\hbar}(f) \mathrm{Op}_{\hbar}\left(\frac{1}{f} \right) (\mathrm{Id} + R_{\hbar})^{-1} = \mathrm{Id}, \]
therefore $\mathrm{Op}_{\hbar}(f)$ is surjective. Similarly, there exists a bounded operator $S_{\hbar}$ with norm $\mathcal{O}(\hbar)$ such that 
$\mathrm{Op}_{\hbar}\left(\frac{1}{f} \right) \mathrm{Op}_{\hbar}(f) = \mathrm{Id} + S_{\hbar}$
and there exists $\hbar_2 \in I$ such that for every $\hbar \leq \hbar_{2}$, $\mathrm{Id} + S_{\hbar}$ is invertible, so
\[ ( \mathrm{Id} + S_{\hbar})^{-1} \mathrm{Op}_{\hbar}\left(\frac{1}{f} \right) \mathrm{Op}_{\hbar}(f) = \mathrm{Id}\]
and hence $\mathrm{Op}_{\hbar}(f)$ is injective. Consequently, $\mathrm{Op}_{\hbar}(f)$ is bijective for every $\hbar \leq \hbar_0 := \min(\hbar_1,\hbar_2)$. Since $\mathrm{Op}_{\hbar}(f)$ is a bounded operator, the inverse mapping theorem \cite[Theorem III.11]{ReedSimon} implies that it is invertible for every $\hbar \leq \hbar_0$. It remains to show that the norm of its inverse is uniformly bounded in $\hbar$. For this we notice that, by Corollary \ref{cor:bounded}, $\mathrm{Op}_{\hbar}(1/f)$ is bounded since $1/f$ is bounded, and we have the inequality
\[ \left\| \mathrm{Op}_{\hbar}(f)^{-1} \right\| \leq \left\| \mathrm{Id} + S_{\hbar} \right\|^{-1} \left\|\mathrm{Op}_{\hbar}\left(\frac{1}{f}\right) \right\| \leq \left\|\frac{1}{f}\right\|_{\infty} + \mathcal{O}(\hbar).   \]
This implies that the norm of $\mathrm{Op}_{\hbar}(f)^{-1}$ is uniformly bounded in $\hbar$.
\end{proof}

\begin{remark}
\label{rmk:norm_inverse}
Note that as a byproduct of the proof of the second point of the corollary, we have that if $f$ is bounded and $|f| \geq c$ for some $c > 0$, then 
$\left\| \mathrm{Op}_{\hbar}(f)^{-1} - \mathrm{Op}_{\hbar}\left(\frac{1}{f}\right) \right\| = \mathcal{O}(\hbar).$
\end{remark}

{We will need one additional axiom for the proof of our main result (Theorem \ref{thm:CH_simpleF}); it may seem quite strong but we do not know how to proceed without it. This axiom is 
\begin{enumerate}[label=(Q7)]
\item \label{item:func_calc} if $g \in \classe{\infty}(\R,\R)$ is bounded, then for every bounded and real-valued $f \in \mathcal{A}_0$, the operator $g(\mathrm{Op}_{\hbar}(f))$ (defined using functional calculus for self-adjoint operators, see e.g. \cite[Theorem VII.1]{ReedSimon}) is such that 
$\| g(\mathrm{Op}_{\hbar}(f)) - \mathrm{Op}_{\hbar}(g \circ f) \| = \mathcal{O}(\hbar)$ (\emph{functional calculus}).
\end{enumerate}
Note that it makes sense to talk about $\mathrm{Op}_{\hbar}(g \circ f)$ since $g \circ f$ belongs to $\mathcal{A}_0$ by the properties of the latter. Note also that this axiom is satisfied by Berezin-Toeplitz operators \cite[Proposition 12]{Cha} and pseudo-differential operators \cite[Theorem 8.7]{DimSjo}.}

We now introduce an algebra $\mathcal{A}_{I}$ whose elements are families $f_I = (f_{\hbar})_{\hbar \in I}$ of functions in $\mathcal{A}_0$ of the form
$f_{\hbar} = f_0 + \hbar f_{1,\hbar}$
with $f_0 \in \mathcal{A}_0$ and where the family $(f_{1,\hbar})_{\hbar \in I}$ is uniformly bounded in $\hbar$ and supported in a compact subset of $M$ which does not depend on $\hbar$. If $f_0$ is also compactly supported, we say that $f_{I}$ is compactly supported. We have a map
\[ \mathrm{Op}: \mathcal{A}_I \to \prod_{\hbar \in I} \mathcal{L}(\Hil_{\hbar}),  f_I  = (f_{\hbar})_{\hbar \in I} \mapsto  (\mathrm{Op}_{\hbar}(f_{\hbar}))_{\hbar \in I}.\]

\begin{definition} \label{so}
A \emph{semiclassical operator} is any element of the image of this map. We denote by $\Psi := \mathrm{Op}(\mathcal{A}_I)$ the set of semiclassical operators.
\end{definition}

We want to define a map $\sigma: \Psi \to \mathcal{A}_0$ which associates to $\mathrm{Op}_{\hbar}(f_I)$ the function $f_0 \in \mathcal{A}_0$. However, we need to check that the latter is unique.

\begin{lemma}
The map $\sigma$ is well-defined. Given $T = (T_{\hbar})_{\hbar \in I} \in \Psi$, we call $\sigma(T)$ the \emph{principal symbol} of $T$. 
\end{lemma}

\begin{proof}
This proof already appeared in \cite[Section 4]{PelPolVu} but we recall it here for the sake of completeness. Let $f_I \in \mathcal{A}_I$ be such that $\mathrm{Op}(f_I) = 0$. Since the family $(f_{1,\hbar})_{\hbar \in I}$ is uniformly bounded in $\hbar$, we deduce from Corollary \ref{cor:bounded} that
\begin{equation} \| \mathrm{Op}_{\hbar}(f_{\hbar}) - \mathrm{Op}_{\hbar}(f_0) \| = \mathcal{O}(\hbar). \label{eq:product_bis} \end{equation}
Let $\chi$ be any compactly supported smooth function. Using the previous estimate and axiom \ref{item:product}, we obtain that 
\[ \left\| \mathrm{Op}_{\hbar}(f_{\hbar}) \mathrm{Op}_{\hbar}(\chi) - \mathrm{Op}_{\hbar}(f_{\hbar} \chi)) \right\| = \mathcal{O}(\hbar), \] 
hence $\left\|\mathrm{Op}_{\hbar}(f_{\hbar} \chi)) \right\| = \mathcal{O}(\hbar)$. Consequently, applying Equation (\ref{eq:product_bis}) to $f_{\hbar} \chi$  yields the equality $\left\|\mathrm{Op}_{\hbar}(f_{0} \chi)) \right\| = \mathcal{O}(\hbar)$. Therefore, by axiom \ref{item:non-degeneracy}, we conclude that $f_0 \chi =0$. Since $\chi$ was arbitrary, this means that $f_0 = 0$.
\end{proof}

By axiom \ref{item:normalization}, the principal symbol of the identity is $\sigma(\mathrm{Id}) = 1$. Axiom \ref{item:reality} implies that the principal symbol of a self-adjoint semiclassical operator is real-valued. We can also draw conclusions about the principal symbol of a unitary operator. 

\begin{lemma}
\label{lm:symb_unitary}
The principal symbol of a unitary semiclassical operator is $\mathbb{S}^1$-valued.
\end{lemma}

\begin{proof}
Let $U_{\hbar}$ be a unitary semiclassical operator. Since we are only interested in the principal symbol, we can assume that $U_{\hbar} = \mathrm{Op}_{\hbar}(f)$ for some $f \in \mathcal{A}_0$. Let $m \in M$ and let $\chi \geq 0$ be a smooth compactly supported function such that $\chi(m) = 1$. By axiom \ref{item:product}, we get that
\begin{equation} \left\| \mathrm{Op}_{\hbar}(\chi^2 |f|^2) - \mathrm{Op}_{\hbar}(\chi \bar{f} \ )\mathrm{Op}_{\hbar}(\chi f) \right\| = \mathcal{O}(\hbar). \label{eq:lm_unit}\end{equation}
But, still because of axiom \ref{item:product}, we have that 
$\left\| \mathrm{Op}_{\hbar}(\chi f) - \mathrm{Op}_{\hbar}(f)\mathrm{Op}_{\hbar}(\chi) \right\| = \mathcal{O}(\hbar),$
which yields thanks to Corollary \ref{cor:bounded} applied to $\chi \bar{f}$:
\[ \left\| \mathrm{Op}_{\hbar}(\chi \bar{f} \ )\mathrm{Op}_{\hbar}(\chi f) - \mathrm{Op}_{\hbar}(\chi \bar{f} \ )\mathrm{Op}_{\hbar}(f)\mathrm{Op}_{\hbar}(\chi) \right\| = \mathcal{O}(\hbar). \]
Therefore we obtain by using (\ref{eq:lm_unit}) and the triangle inequality: 
\[ \left\| \mathrm{Op}_{\hbar}(\chi^2 |f|^2) - \mathrm{Op}_{\hbar}(\chi \bar{f} \ )\mathrm{Op}_{\hbar}(f)\mathrm{Op}_{\hbar}(\chi) \right\| = \mathcal{O}(\hbar). \]
By iterating the same method, we eventually get
\[ \left\| \mathrm{Op}_{\hbar}(\chi^2 |f|^2) - \mathrm{Op}_{\hbar}(\chi)\mathrm{Op}_{\hbar}(\bar{f} \ )\mathrm{Op}_{\hbar}(f)\mathrm{Op}_{\hbar}(\chi) \right\| = \mathcal{O}(\hbar). \]
Now, using axiom \ref{item:reality} and the fact that $\mathrm{Op}_{\hbar}(f)$ is unitary, this yields
$\left\| \mathrm{Op}_{\hbar}(\chi^2 |f|^2) - \mathrm{Op}_{\hbar}(\chi)^2 \right\| = \mathcal{O}(\hbar).$
Finally, thanks to axiom \ref{item:product} and the linearity of $\mathrm{Op}_{\hbar}$, we infer from this equality that
\[ \left\| \mathrm{Op}_{\hbar}\left(\chi^2 (|f|^2 - 1) \right) \right\| = \mathcal{O}(\hbar),\]
thus as a consequence of axiom \ref{item:non-degeneracy} we have that $\chi^2 (|f|^2 - 1) = 0$, hence $|f(m)|^2 = 1$. 
\end{proof}

Let us state a final useful remark regarding semiclassical operators. Let $T_{\hbar} \in \Psi$ be a semiclassical operator with bounded principal symbol $f$, and such that $|f| \geq c$ for some $c > 0$. Then as a consequence of Corollary \ref{cor:inverse}, $T_{\hbar}$ is invertible for $\hbar$ sufficiently small. Indeed, $\mathrm{Op}_{\hbar}(f)$ is invertible and $T_{\hbar} = \mathrm{Op}_{\hbar}(f) + \mathcal{O}(\hbar);$
thus our claim comes from an application of  \cite[Theorem A.3.30]{HislopSigal}.

\section{Proof of the main theorem} \label{sec:proofofmaintheorem}

In this section we prove our main result, which we shall reformulate here in precise terms as Theorem \ref{thm:CH_simpleF}.  
In the proof we use the results proved in \cite{PelPolVu} for the self-adjoint case.

\subsection{Cayley transform}
\label{sect:cayley}

Let us recall the definition of the inverse Cayley transform of a unitary operator \cite[Definition 3.17]{Rudin}. Let $U$ be a unitary operator such that $-1 \notin \text{Sp}(U)$. We define the inverse Cayley transform of $U$ as
$\mathcal{C}(U) = i(\mathrm{Id} - U)(\mathrm{Id} + U)^{-1}.$
Then $\mathcal{C}(U)$ is a self-adjoint operator.

{Moreover, using functional calculus, we can define the transform
\begin{equation} \mathcal{T}(U) = 2\arctan(\mathcal{C}(U)) \label{eq:new_transform}\end{equation}
of a unitary operator $U$. We introduce the function
\[ \phi: \mathbb{C} \setminus \{ -1\} \to \mathbb{C}, \quad z \mapsto i\frac{1-z}{1+z}, \]
so that $\mathcal{C}(U) = \phi(U)$. One readily checks that for $z \in \mathbb{S}^1 \setminus \{ -1\}$  
\begin{equation} \phi(z) =  \tan\left( \tfrac{1}{2} \arg z \right). \label{eq:phi_arg}\end{equation}
{This implies that if $U$ is unitary and $-1 \notin \text{Sp}(U)$, $\mathcal{C}(U)$ is bounded; indeed, since $\text{Sp}(U)$ is closed, there exists $\varepsilon > 0$ such that for every $z \in \text{Sp}(U)$, $\arg(z) \in [-\pi + \varepsilon, \pi - \varepsilon]$. Hence $\phi$ is bounded on $\text{Sp}(U)$, and the properties of functional calculus imply that $\mathcal{C}(U) = \phi(U)$ is bounded.}
Now, we consider the principal value of $\arctan$, given by the formula
\[ \arctan(z) = \frac{i}{2} \left( \Log(1-iz) - \Log(1+iz) \right) \]
where $\Log$ is the principal value of the complex logarithm. It is holomorphic in $\mathbb{C} \setminus \left( i[1,+\infty) \cup i(-\infty,-1) \right)$, and we define the function $\psi$ in a neighborhood of $\mathbb{S}^1 \setminus \{ -1\}$ as 
\begin{equation} \psi(z) = 2 \arctan(\phi(z)), \label{eq:func_psi}\end{equation}
so that $\mathcal{T}(U) = \psi(U)$ for every unitary operator $U$, and $\psi(z) = \arg(z)$ whenever $z$ belongs to $\mathbb{S}^1 \setminus \{ -1\}$}

{
\begin{lemma}
\label{lemma:comm_transform}
Let $U, V$ be commuting unitary operators acting on a Hilbert space $\Hil$, none of them having $-1$ in its spectrum. Then $\mathcal{T}(U)$ and $\mathcal{T}(V)$ commute.
\end{lemma}}

{\begin{proof}
This is a consequence of the following fact: if $A$ is a normal operator acting on a Hilbert space $\Hil$, with spectral measure $E_A$, $S$ is a Borel set and $f:\mathbb{C} \to \mathbb{C}$ is a measurable function, then 
$E_{f(A)}(S) = E_A(f^{-1}(S)).$
Therefore, if $B$ is another normal operator which commutes with $A$ and $g$ is another measurable function, the spectral projections $E_{f(A)}(S)$ and $E_{g(B)}(T)$ commute for every Borel sets $S, T$. Hence $f(A)$ and $g(B)$ commute.
\end{proof}}

{Consequently, if $U_1, \ldots, U_d$ are commuting unitary operators, it makes sense to talk about the joint spectrum of the family $\mathcal{T}(U_1), \ldots, \mathcal{T}(U_d)$.} We recall that the joint spectrum of a finite family of pairwise commuting normal operators is defined as the support of its joint spectral measure.

{\begin{lemma}
\label{lemma:jspec_transform}
Let $U_1, \ldots, U_d$ be commuting unitary operators acting on a Hilbert space $\Hil$, 
none of them having $-1$ in its spectrum. Then 
\[ \mathrm{JointSpec}(\mathcal{T}(U_1), \ldots, \mathcal{T}(U_d)) = \overline{\left\{  \arg(\lambda), \ \lambda \in \mathrm{JointSpec}(U_1, \ldots, U_d) \right\}}. \]
\end{lemma}}

{\begin{proof}
We mimic the reasoning of the proof of Proposition $5.25$ in \cite{Schmu} (which deals with the spectrum of one single operator). For every $j \in \llbracket 1, d \rrbracket$, we have that $\mathcal{T}(U_j) = \psi(U_j)$, see Equation (\ref{eq:func_psi})
Let $\mu = E_{U_1} \otimes \ldots \otimes E_{U_d}$
be the joint spectral measure of $U_1, \ldots, U_d$, and let 
$\nu = E_{\mathcal{T}(U_1)} \otimes \ldots \otimes E_{\mathcal{T}(U_d)}$
be the joint spectral measure of $\mathcal{T}(U_1), \ldots, \mathcal{T}(U_d)$; we need to prove that 
$ \mathrm{supp}(\nu) = \overline{\left\{(\psi(\lambda_1), \ldots, \psi(\lambda_d)),  \ \lambda \in \mathrm{supp}(\mu) \right\}} =: S$.
\
Firstly, let $\zeta = (\zeta_1, \ldots, \zeta_d) \in S$, and let $\varepsilon_1, \ldots, \varepsilon_d > 0$ be small enough; there exists $\lambda = (\lambda_1, \ldots, \lambda_d) \in \mathrm{supp}(\mu)$ such that for every $j \in \llbracket 1, d \rrbracket$, the inequality 
$|\zeta_j - \lambda_j| < \varepsilon_j$
holds. Since $\psi$ is continuous in a neighborhood of $\mathrm{Sp}(U_j)$ (because $\mathrm{Sp}(U_j)$ is closed and does not contain $-1$), there exists $\delta_j > 0$ such that  
\[ D(\lambda_j,\delta_j) \subset \left\{ z \in \mathbb{C}, \ |\psi(z) - \psi(\lambda_j)| < \varepsilon_j \right\} \subset \psi^{-1}\left( D(\zeta_j, 2 \varepsilon_j) \right) \]
where $D(z,r)$ stands for the open disk of radius $r$ centered at $z$. We deduce from this inclusion that
$E_{U_j}\left(\psi^{-1}\left( D(\zeta_j, 2 \varepsilon_j) \right)\right) \geq E_{U_j}\left( D(\lambda_j,\delta_j) \right) > 0$,
where the last inequality comes from the fact that $\lambda$ belongs to the support of $E_{U_j}$. Consequently, if 
$D := D(\zeta_1, 2 \varepsilon_1) \times \ldots \times D(\zeta_d, 2 \varepsilon_d),$
we have that 
\[ \nu(D) = \prod_{j=1}^d E_{\psi(U_j)}(D(\zeta_j, 2 \varepsilon_j)) = \prod_{j=1}^d E_{U_j}(\psi^{-1}\left( D(\zeta_j, 2 \varepsilon_j) \right)) > 0, \]
which means that $\zeta$ belongs to the support of $\nu$.
\
Conversely, if $\zeta \notin S$, there exists $j \in \llbracket 1, d \rrbracket$ such that $\psi^{-1}\left( D(\zeta_j, \varepsilon_j) \right) $ is empty for every $\varepsilon_j > 0$ small enough, and we conclude with similar computations that $\zeta \notin \mathrm{supp}(\nu)$.
\end{proof}}

We need the following technical tool for the proof of the main theorem.

\begin{lemma}
\label{thm:PelPolVu_modif}
Let $A_1(\hbar), \ldots, A_d(\hbar)$ be pairwise commuting self-adjoint operators acting on $\Hil_{\hbar}$, and let $T_1(\hbar), \ldots, T_d(\hbar)$ be self-adjoint semiclassical operators acting on $\Hil_{\hbar}$, with bounded principal symbols $f_1, \ldots, f_d$. Assume moreover that for all $j \in \llbracket 1,d \rrbracket$,  $\| T_j(\hbar) - A_j(\hbar) \| = \mathcal{O}(\hbar).$ Then 
\[ \mathrm{Convex \ Hull}\left( \mathrm{JointSpec}(A_1(\hbar),\ldots,A_d(\hbar)) \right) \underset{\hbar \to 0}{\longrightarrow} \overline{\mathrm{Convex \ Hull}\left( F(M) \right)} \]
where $F=(f_1, \ldots, f_d): M \to \R^d$.
\end{lemma}
Since the proof is close to the one of the aforementioned theorem, we will assume some degree of familiarity with the content of \cite{PelPolVu}. The first step is to prove the following result comparing only two operators.
\begin{lemma}
Let $A_{\hbar}, T_{\hbar}$ be self-adjoint operators acting on $\Hil_{\hbar}$ such that $T_{\hbar}$ is a semiclassical operator with principal symbol $f_0$ and $\| A_{\hbar} - T_{\hbar} \| = \mathcal{O}(\hbar)$. Let $\lambda_{\sup}(\hbar) = \sup \mathrm{Sp}(A_{\hbar})$, which may be infinite. Then
$\lambda_{\sup}(\hbar) \underset{\hbar \to 0}{\longrightarrow} \sup_M f_0.$
\end{lemma}

\begin{proof}
Let $R_{\hbar} = A_{\hbar} - T_{\hbar}$, so that $\|R_{\hbar}\| = \mathcal{O}(\hbar)$ by assumption; we choose $\hbar_0 \in I$ and $C > 0$ such that $\|R_{\hbar}\| \leq C \hbar$ for every $\hbar \leq \hbar_0$. We also introduce the (possibly infinite) quantity $\mu_{\sup}(\hbar) = \sup \mathrm{Sp}(T_{\hbar})).$
Our goal is to compare $\lambda_{\sup}(\hbar)$ to $\mu_{\sup}(\hbar)$; of course, thanks to Equation (\ref{eq:rayleigh}), we have that
\[ \lambda_{\sup}(\hbar) = \sup_{v \in \Hil_{\hbar}, \|v\|=1} \scal{A_{\hbar}v}{v}, \quad \mu_{\sup}(\hbar) = \sup_{v \in \Hil_{\hbar}, \|v\|=1} \scal{T_{\hbar}v}{v}.\]
Let $\hbar \leq \hbar_0$, and let us start with the case where $\mu_{\sup}(\hbar) = +\infty$. Let $M > 0$; there exists $v_0 \in \Hil_{\hbar}$ with unit norm such that $\scal{T_{\hbar}v_0}{v_0} \geq M$; this yields
\[ \scal{A_{\hbar}v_0}{v_0} = \scal{T_{\hbar}v_0}{v_0} + \scal{R_{\hbar}v_0}{v_0} \geq M - C \hbar.\]
Since $M$ is arbitrarily large, this means that $\lambda_{\sup}(\hbar) = +\infty$.
Now, we assume that $\mu_{\sup}(\hbar)$ is finite. From the equality
\[ \lambda_{\sup}(\hbar) = \sup_{v \in \Hil_{\hbar}, \|v\|=1} \left( \scal{T_{\hbar}v}{v} + \scal{R_{\hbar}v}{v} \right),\]
we derive that
\[ \lambda_{\sup}(\hbar) \leq \mu_{\sup}(\hbar) +  \sup_{v \in \Hil_{\hbar}, \|v\|=1} \scal{R_{\hbar}v}{v} \leq \mu_{\sup}(\hbar) + C \hbar. \]
Moreover, there exists a unit vector $v_0 \in \Hil$ such that 
$\mu_{\sup}(\hbar) \leq \scal{T_{\hbar}v_0}{v_0} + \hbar.$
By decomposing
\[ \scal{A_{\hbar}v_0}{v_0} = \scal{R_{\hbar}v_0}{v_0} + \scal{T_{\hbar}v_0}{v_0} - \mu_{\sup}(\hbar) + \mu_{\sup}(\hbar),\]
we get that 
$\lambda_{\sup}(\hbar) \geq \scal{A_{\hbar}v_0}{v_0} \geq \mu_{\sup}(\hbar) -(C+1)\hbar,$
so finally
\[ \mu_{\sup}(\hbar) -(C+1)\hbar \leq \lambda_{\sup}(\hbar) \leq  \mu_{\sup}(\hbar) + C \hbar. \]
Therefore, the result comes from the fact that $\mu_{\sup}(\hbar)$ tends to $\sup_M f_0$ as $\hbar$ goes to zero \cite[Lemma 11]{PelPolVu}.
\end{proof}

\begin{proof}[Proof of Lemma \ref{thm:PelPolVu_modif}]
We follow the reasoning of the proof of \cite[Theorem~8]{PelPolVu}. More precisely, let $\Sigma_{\hbar} =  \mathrm{JointSpec}(A_1(\hbar),\ldots,A_d(\hbar))$ and consider, for any subset $S$ of $\R^d$, the function 
\[ \Phi_S: \mathbb{S}^{d-1} \to \R \cup \{+\infty\}, \quad \alpha \mapsto \sup_{x \in S} \sum_{j=1}^d \alpha_j  x_j.\]
Then it suffices to show that $\Phi_{\Sigma_{\hbar}}$ converges uniformly to $\Phi_{F(M)}$ as $\hbar$ goes to zero. We start by proving the pointwise convergence. Let $\alpha \in \mathbb{S}^{d-1}$ and consider the self-adjoint operator $A^{(\alpha)}_{\hbar} = \sum_{j=1}^d \alpha_j A_j(\hbar)$; by \cite[Lemma~14]{PelPolVu}, $\Phi_{\Sigma_{\hbar}}(\alpha) = \sup \mathrm{Sp}(A^{(\alpha)}_{\hbar})$. In a similar fashion, we introduce the operator $T^{(\alpha)}_{\hbar} = \sum_{j=1}^d \alpha_j T_j(\hbar)$ and the function $f^{(\alpha)} =  \sum_{j=1}^d \alpha_j f_j$, so that $T^{(\alpha)}_{\hbar}$ is a self-adjoint semiclassical operator with principal symbol $f^{\alpha}$. Furthermore, since $\| T_j(\hbar) - A_j(\hbar) \| = \mathcal{O}(\hbar)$ for $j=1, \ldots, d$, we also have the estimate $\| T^{(\alpha)}(\hbar) - A^{(\alpha)}(\hbar) \| = \mathcal{O}(\hbar)$. Consequently, it follows from the previous lemma that 
\[ \Phi_{\Sigma_{\hbar}}(\alpha) = \sup \mathrm{Sp}\left(A^{(\alpha)}_{\hbar}\right) \underset{\hbar \to 0}{\longrightarrow} \sup_M f^{(\alpha)} = \Phi_{F(M)}(\alpha). \]
To prove that this convergence is uniform, we observe that the boundedness of the principal symbols $f_1, \ldots f_d$ implies the boundedness of $T_1(\hbar), \ldots T_d(\hbar)$, which in turn implies the boundedness of $A_1(\hbar), \ldots A_d(\hbar)$. Therefore the joint spectrum of the latter family is bounded, hence compact. We conclude by the argument used in the last part of the proof of Theorem 8 in \cite{PelPolVu}.
\end{proof}

\subsection{If no principal symbol is onto}
\label{subsect:very_simple}

In this section, we consider pairwise commuting unitary semiclassical operators $U_1(\hbar), \ldots, U_d(\hbar)$ with joint principal symbol $F = (f_0^1, \ldots, f_0^d)$. We assume that for every $j \in \llbracket 1,d \rrbracket$, $f_0^j(M)$ is closed, and that the same holds for $F(M)$. 
We assume moreover that none of the principal symbols $f_0^j: M \to \mathbb{S}^1$, $ j \in \llbracket 1,d \rrbracket,$
is onto; using the terminology introduced earlier, this means that $F(M)$ is a simple compact subset of $\T^d$. Note that this set is connected since it is the image of $M$, which is itself connected, by a continuous function.

Let us introduce an additional assumption in the case where the joint spectrum of $(U_1(\hbar), \ldots, U_d(\hbar))$ is generic (see Lemma \ref{lm:generic}): 

\begin{enumerate}[label=(A\arabic*)]
\item \label{item:hyp_generic} There {exist} $\hbar_0 \in I$ and a point $b \in \T^d$ which is admissible (see Lemma \ref{lm:diam} for the terminology) for all $\mathrm{JointSpec}(U_1(\hbar), \ldots, U_d(\hbar))$, $\hbar \leq \hbar_0,$
and such that $b \cdot F(M)$ is very simple.
\end{enumerate}

\begin{remark}
This assumption might seem strange but will be crucial for a part of our analysis. Indeed, it may not hold if the joint spectrum is too sparse (see Figure \ref{fig:assumption}). In this situation, given the data of the joint spectrum only, its convex hull computed thanks to our definition will be far from the convex hull of $F(M)$. However, this assumption is reasonable, because it holds for Berezin-Toeplitz and pseudodifferential operators, as a corollary of the Bohr-Sommerfeld rules which imply that the joint spectrum is ``dense'' (when $\hbar \to 0$) in the set of regular values of $F$ (see \cite{HelRob} for pseudodifferential operators and \cite{Cha_BS} for Berezin-Toeplitz operators). Nevertheless, our assumption is much weaker than the Bohr-Sommerfeld rules.
\end{remark}

\begin{figure}[h]
\subfigure[$\arg(E)$]{\includegraphics{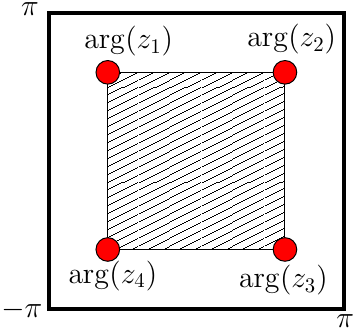}}
\hspace{3mm}
\subfigure[$\arg(b \cdot E)$]{\includegraphics{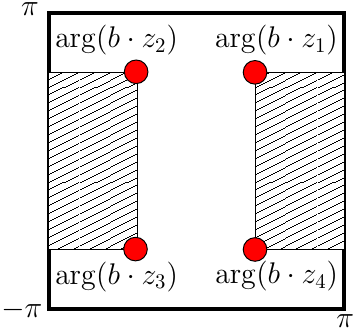}}
\caption{An example for which assumption {\rm \ref{item:hyp_generic}} does not hold.}
\label{fig:assumption}
\end{figure}

Now, we do not necessarily assume that the joint spectrum is generic anymore. Our goal is to prove the following result.

\begin{theorem}
\label{thm:CH_simpleF}
For every $b \in \T^d$ such that $b \cdot F(M)$ is very simple, 
\[ b^{-1} \cdot \exp \left(i \ \mathrm{Convex \ Hull}(\arg(b \cdot \mathrm{JointSpec}(U_1(\hbar), \ldots, U_d(\hbar)))) \right)  \]
converges, when $\hbar \to 0$, to 
$\overline{\mathrm{Convex \ Hull}_{\T^d}(F(M))}$
with respect to the Hausdorff distance on $\T^d$. In particular, if the joint spectrum is generic and assumption 
{\rm \ref{item:hyp_generic}} holds, then
\begin{equation} \mathrm{Convex \ Hull}_{\mathbb{T}^d}(\mathrm{JointSpec}(U_1(\hbar), \ldots, U_d(\hbar))) \underset{\hbar \to 0}{\longrightarrow} \overline{\mathrm{Convex \ Hull}_{\mathbb{T}^d}(F(M))}. \end{equation}
\end{theorem}

In this statement, we use that $b \cdot \mathrm{JointSpec}(U_1(\hbar), \ldots, U_d(\hbar)))$ is very simple, for $\hbar \in I$ small enough, whenever $b \cdot F(M)$ is very simple. This is a consequence of the following lemma.
\begin{lemma}
\label{lm:jsp_simple}
Let $j$ in $\llbracket 1,d \rrbracket$, and let $a \in \mathbb{S}^1 \setminus f_0^j(M)$. Then there exists $\hbar_0 \in I$ such that for every $\hbar \leq \hbar_0$ in $I$, $a \notin \mathrm{Sp}(U_j(\hbar))$.
\end{lemma}

\begin{proof}
This is a consequence of Corollary \ref{cor:inverse} (more precisely, of its consequence stated right after the proof of Lemma \ref{lm:symb_unitary}). Indeed, since $f_0^j(M)$ is closed, there exists a small open neighborhood of $a$ in $\mathbb{S}^1$ not intersecting it. Thus there exists $c > 0$ such that 
$|f_0^j - a| \geq c.$
Hence the operator $U_j(\hbar) - a \ \mathrm{Id}$ is invertible, so $a$ does not belong to the spectrum of $U_j(\hbar)$.
\end{proof}

\begin{lemma}
\label{lm:diff_cayley_op}
Let $(U_{\hbar})_{\hbar \in I}$ be a unitary semiclassical operator with principal symbol $f_0: M \to \mathbb{S}^1$ such that $f_0(M)$ is closed and does not contain $-1$. Then 
$\left\| \mathcal{C}(U_{\hbar}) - \mathrm{Op}_{\hbar}(\phi \circ f_0) \right\| = \mathcal{O}(\hbar),$
where the function $\phi: \mathbb{C} \setminus \{-1\} \to \mathbb{C}$ is defined as $\phi(z) = i\frac{1-z}{1+z}$.
\end{lemma}

Note that this statement makes sense since by the above lemma, there exists $\hbar_0 \in I$ such that for every $\hbar \leq \hbar_0$, $-1 \notin \mathrm{Sp}(U_{\hbar})$. {Moreover, $\phi \circ f = i\frac{1-f}{1+f}$ belongs to $\mathcal{A}_0$ by the properties of this algebra, hence it makes sense to introduce $\mathrm{Op}_{\hbar}(\phi \circ f_0)$.}

\begin{proof}
By the same argument that we have used in the proof of the previous lemma, $\mathrm{Op}_{\hbar}(1 + f_0)$ is invertible and the norm of its inverse is uniformly bounded in $\hbar$. Thus by axiom \ref{item:composition} and Remark \ref{rmk:norm_inverse}, we have that
\[ \left\| \mathrm{Op}_{\hbar}(\phi \circ f_0) - i \ \mathrm{Op}_{\hbar}(1-f_0) \mathrm{Op}_{\hbar}(1+f_0)^{-1} \right\| = \mathcal{O}(\hbar), \]
which yields by axiom \ref{item:normalization}:
\begin{equation}\label{eq:diff} \left\| \mathrm{Op}_{\hbar}(\phi \circ f_0) - i \ (\mathrm{Id} - \mathrm{Op}_{\hbar}(f_0)) (\mathrm{Id} + \mathrm{Op}_{\hbar}(f_0))^{-1} \right\| = \mathcal{O}(\hbar). \end{equation}
Furhermore, $\mathrm{Id} + U_{\hbar} = \mathrm{Id} + \mathrm{Op}_{\hbar}(f_0) + R_{\hbar}$ with $\| R_{\hbar} \| = \mathcal{O}(\hbar)$. Consequently (see e.g. \cite[Theorem A3.31]{HislopSigal}),
$\left(\mathrm{Id} + U_{\hbar}\right)^{-1} = \left( \mathrm{Id} + \mathrm{Op}_{\hbar}(f_0) \right)^{-1}  \left( \mathrm{Id} + A_{\hbar} \right)^{-1}$
where
$A_{\hbar} = R_{\hbar}  \left( \mathrm{Id} + \mathrm{Op}_{\hbar}(f_0) \right)^{-1};$
observe that $\|A_{\hbar}\| = \mathcal{O}(\hbar)$. We derive from the above equation the inequality
\[ \left\| \left(\mathrm{Id} + U_{\hbar}\right)^{-1} -  \left( \mathrm{Id} + \mathrm{Op}_{\hbar}(f_0) \right)^{-1} \right\| \leq \left\| \left( \mathrm{Id} + \mathrm{Op}_{\hbar}(f_0) \right)^{-1} \right\|  \left\| \left( \mathrm{Id} + A_{\hbar} \right)^{-1} - \mathrm{Id} \right\|. \]
But we have that 
\[  \left\| \left( \mathrm{Id} + A_{\hbar} \right)^{-1} - \mathrm{Id} \right\| \leq \sum_{n=1}^{+\infty} \| A_{\hbar} \|^n = \frac{1}{1-\|A_{\hbar}\|} -1 = \mathcal{O}(\hbar). \]
Therefore we finally obtain that 
$\| \left(\mathrm{Id} + U_{\hbar}\right)^{-1} -  \left( \mathrm{Id} + \mathrm{Op}_{\hbar}(f_0) \right)^{-1} \| = \mathcal{O}(\hbar).$
Since obviously
$\| (\mathrm{Id} - U_{\hbar}) -  ( \mathrm{Id} - \mathrm{Op}_{\hbar}(f_0)) \| = \mathcal{O}(\hbar),$
Equation (\ref{eq:diff}) and the triangle inequality finally yield
\[ \left\| \mathrm{Op}_{\hbar}(\phi \circ f_0) - i \ (\mathrm{Id} - U_{\hbar}) (\mathrm{Id} + U_{\hbar})^{-1} \right\| = \mathcal{O}(\hbar),\]
which was to be proved. 
\end{proof}

{\begin{lemma}
\label{lm:diff_transform_op}
Let $(U_{\hbar})_{\hbar \in I}$ be a unitary semiclassical operator with principal symbol $f_0: M \to \mathbb{S}^1$ such that $f_0(M)$ is closed and does not contain $-1$. Then, if $\psi$ is as in Equation (\ref{eq:func_psi}),
\[ \left\| \mathcal{T}(U_{\hbar}) - \mathrm{Op}_{\hbar}(\psi \circ f_0) \right\| = \mathcal{O}(\hbar). \]
\end{lemma}}

{Note that it makes sense to talk about the operator $\mathrm{Op}_{\hbar}(\psi \circ f_0) =  \text{Op}_{\hbar}(2 \arctan \circ \phi \circ f_0)$ since $\arctan \in \classe{\infty}{(\R,\R)}$ is bounded and $\phi \circ f_0$ is real-valued (because of Equation (\ref{eq:phi_arg})) and bounded (because, as obtained in the proof of Lemma \ref{lm:jsp_simple}, $(\arg \circ f_0)(M)$ is contained in an interval of the form $[-\pi+\varepsilon,\pi-\varepsilon]$, with $\varepsilon > 0$).}

{\begin{proof}
Let $A_{\hbar} = \mathcal{C}(U_{\hbar})$. The proof of Lemma \ref{lm:jsp_simple} can be adapted to show that there exists $c > 0$ such that $\text{Sp}(A_{\hbar}) \subset [-c,c]$ for every $\hbar$ sufficiently small. Since $A_{\hbar}$ is bounded and $G = 2 \arctan$ is holomorphic in a neighborhood of $[-c,c]$, we can use holomorphic functional calculus and write
\[ \psi(U_{\hbar}) = G(A_{\hbar}) = \frac{1}{2i\pi} \int_{\Gamma} G(\zeta) \left( \zeta - A_{\hbar} \right)^{-1} \ d\zeta, \]
where $\Gamma$ is a positively oriented contour containing $[-c,c]$ in its interior (see for instance \cite[Section VII.9]{DunSch}). By the previous lemma, we know that $A_{\hbar} = \text{Op}_{\hbar}(\phi \circ f_0) + R_{\hbar}$ with $\| R_{\hbar} \| = \mathcal{O}(\hbar)$. Hence
\[ \begin{split}  \left( \zeta - A_{\hbar} \right)^{-1} &  =  \left( \zeta - \text{Op}_{\hbar}(\phi \circ f_0) - R_{\hbar} \right)^{-1} \\
& =  \left(\text{Id} -  \underbrace{(\zeta - \text{Op}_{\hbar}(\phi \circ f_0))^{-1} R_{\hbar}}_{B_{\hbar}} \right)^{-1}  \left( \zeta - \text{Op}_{\hbar}(\phi \circ f_0) \right)^{-1}. \end{split} \]
We claim that $\left(\text{Id} - B_{\hbar} \right)^{-1} = \text{Id} + \mathcal{O}(\hbar)$ uniformly in $\zeta$; indeed, we have that 
\[ \| B_{\hbar} \| \leq \left\| (\zeta - \text{Op}_{\hbar}(\phi \circ f_0))^{-1} \right\| \ \| R_{\hbar} \| \leq \frac{\| R_{\hbar} \|}{d(\zeta,\text{Sp}(\text{Op}_{\hbar}(\phi \circ f_0)))}, \]
see for instance \cite[Theorem 5.8]{HislopSigal} for the last equality. Since the distance $d(\zeta,\text{Sp}(\text{Op}_{\hbar}(\phi \circ f_0)))$ is bounded from below uniformly in $\zeta$, this yields $\| B_{\hbar} \| = \mathcal{O}(\hbar)$ uniformly in $\zeta$, and we obtain as in the proof of the previous lemma that
\[  \left\| \left( \mathrm{Id} + B_{\hbar} \right)^{-1} - \mathrm{Id} \right\| \leq \frac{1}{1-\|B_{\hbar}\|} -1 = \mathcal{O}(\hbar) \]
uniformly in $\zeta$. Consequently, since $G$ is bounded on $[-c,c]$,
\[ G(A_{\hbar}) = \frac{1}{2i\pi} \int_{\Gamma} G(\zeta) \left( \zeta - \text{Op}_{\hbar}(\phi \circ f_0) \right)^{-1} \ d\zeta + \mathcal{O}(\hbar) = G\left(\text{Op}_{\hbar}(\phi \circ f_0)\right) + \mathcal{O}(\hbar). \]
Now, it follows from Axiom \ref{item:func_calc} that
\[ G\left(\text{Op}_{\hbar}(\phi \circ f_0)\right) = \text{Op}_{\hbar}(G \circ \phi \circ f_0) + \mathcal{O}(\hbar) = \text{Op}_{\hbar}(\psi \circ f_0)+ \mathcal{O}(\hbar); \] 
indeed, the function $\arctan: \R \to \R$ is bounded and $\phi \circ f_0$ is real-valued and bounded. Consequently, we obtain that 
\[ G(A_{\hbar}) = \text{Op}_{\hbar}(\psi \circ f_0) + \mathcal{O}(\hbar). \]
\end{proof}}

Before proving Theorem \ref{thm:CH_simpleF}, we state one last technical lemma.

\begin{lemma}
\label{lm:cv}
Let $E$ be a compact subset of $(-\pi,\pi)^d$ and let $(E_{\varepsilon})_{\varepsilon > 0}$ be a family of compact subsets of $(-\pi,\pi)^d$ such that 
$d_H\left(E, E_{\varepsilon} \right) \underset{\varepsilon \to 0}{\longrightarrow} 0.$
Then $d_H^{\T^d}\left(\exp(iE), \exp(i E_{\varepsilon}) \right) \underset{\varepsilon \to 0}{\longrightarrow} 0$.
\end{lemma}

\begin{proof}
Let $\delta_0 = d(E,\partial([-\pi,\pi]^d))$
be the distance between $E$ and the boundary of $[-\pi,\pi]^d$ in $\R^d$. Choose a positive number $\delta \leq \tfrac{1}{2} \delta_0$; there exists $\varepsilon > 0$ such that $d_H(E,E_{\varepsilon}) \leq \delta$. Let $\gamma$ be such that $1 < \gamma < 2$. Let $u \in E$; by definition of the Hausdorff distance, there exists $v \in E_{\varepsilon}$ such that 
$\| u - v \|_{\R^d} \leq \gamma d_H(E,E_{\varepsilon}).$
Now, let $\theta \in (2\pi\Z)^d$ be non-zero; then $v - \theta$ does not belong to $[-\pi,\pi]^d$, thus 
$\| u - v + \theta \|_{\R^d} \geq \delta_0 \geq \gamma \delta \geq \| u - v \|_{\R^d}.$
Consequently, we have that
\[ d^{\T^d}\left(\exp(iu), \exp(iv) \right) = \| u - v \|_{\R^d} \leq  \gamma d_H(E,E_{\varepsilon}) \leq \gamma \delta.\]
Therefore $d^{\T^d}\left(\exp(iu), \exp(i E_{\varepsilon}) \right) \leq  \gamma \delta.$ Exchanging the roles of $E$ and $E_{\varepsilon}$, we also get that for every $v$ in $E_{\varepsilon}$, 
$d^{\T^d}\left(\exp(iv),  \exp(i E) \right) \leq  \gamma \delta.$
This implies that $d_H^{\T^d}\left(\exp(iE), \exp(i E_{\varepsilon}) \right) \leq \gamma \delta$, because of 
the characterization (\ref{eq:char_haus}). 
\end{proof}

We are finally ready to give proof of the main result of this section.

{\begin{proof}[Proof of Theorem \ref{thm:CH_simpleF}]
Let $b = (b_1, \ldots, b_d) \in \T^d$ be such that $b \cdot F(M)$ is very simple. For every $j \in \llbracket 1, d \rrbracket$, we consider the operator 
$V_j(\hbar) = b_j U_j(\hbar),$
which is a semiclassical unitary operator, with principal symbol 
$g_0^j = b_j f_0^j.$
By Lemma \ref{lm:jsp_simple}, there exists $\hbar_j \in I$ such that $-1 \notin \mathrm{Sp}(V_j(\hbar))$ whenever $\hbar \leq \hbar_j$. Let 
$\hbar_0 = \min_{1 \leq j \leq d} \hbar_j;$
in the rest of the proof we will assume that $\hbar \leq \hbar_0$. We can therefore consider the self-adjoint operators
\[ T_j(\hbar) = \mathcal{T}(V_j(\hbar)), \quad 1 \leq j \leq d, \]
see Equation (\ref{eq:new_transform}) for the definition of $\mathcal{T}$. By Lemma \ref{lemma:comm_transform}, $T_j(\hbar)$ and $T_m(\hbar)$ commute for every $j,m \in \llbracket 1,d \rrbracket$. We also consider the self-adjoint semiclassical operators
$B_j(\hbar) = \mathrm{Op}_{\hbar}(a_0^j)$, $1 \leq j \leq d,$
where $a_0^j = \psi \circ g_0^j$, see Equation (\ref{eq:func_psi}) for the definition of $\psi$.
We also recall that for $z \in \mathbb{S}^1 \setminus \{-1\}$, $\psi(z) = \arg z$, and thus 
$a_0^j = \arg g_0^j = \arg (b_j f_0^j).$
Let $A = (a_0^1,\ldots,a_0^d)$. Since by Lemma \ref{lm:diff_transform_op},
$\| T_j(\hbar) - B_j(\hbar) \| = \mathcal{O}(\hbar)$, for every $j \in \llbracket 1,d \rrbracket$, Lemma~\ref{thm:PelPolVu_modif} implies that
\begin{equation} \mathrm{Convex \ Hull}(\mathrm{JointSpec}(T_1(\hbar), \ldots, T_d(\hbar))) \underset{\hbar \to 0}{\longrightarrow} \overline{\mathrm{Convex \ Hull}(A(M))} \label{eq:CH} \end{equation}
with respect to the Hausdorff distance on $\R^d$. On the one hand, we have the equality $\mathrm{Convex \ Hull}(A(M)) = \mathrm{Convex \ Hull}(\arg(b \cdot F(M))).$ 
 On the other hand, Lemma \ref{lemma:jspec_transform} yields 
\[ \mathrm{JointSpec}(T_1(\hbar), \ldots, T_d(\hbar))) = \overline{\arg(\mathrm{JointSpec}(V_1(\hbar), \ldots, V_d(\hbar)))}. \]
Substituting these results in equation (\ref{eq:CH}), we obtain that
\[ \mathrm{Convex \ Hull}(b \cdot \arg(\mathrm{JointSpec}(U_1(\hbar), \ldots, U_d(\hbar)))) \]
converges, when $\hbar$ goes to zero, to
$\overline{\mathrm{Convex \ Hull}(b \cdot \arg(F(M)))}$
with respect to the Hausdorff distance on $\R^d$. By Lemma \ref{lm:cv}, this in turn implies that 
$$\exp\left(i \ \mathrm{Convex \ Hull}(b \cdot \arg(\mathrm{JointSpec}(U_1(\hbar), \ldots, U_d(\hbar))))\right)$$
converges to
$\exp\left(i \ \overline{\mathrm{Convex \ Hull}(b \cdot \arg(F(M)))} \right)$
for the Hausdorff distance on $\T^d$ when $\hbar$ goes to zero. Using the continuity of $\exp$ and of the restriction of $\arg$ to $(-\pi, \pi)^d$, we see that the latter is
$\overline{ \exp\left(i \ \mathrm{Convex \ Hull}(b \cdot \arg(F(M))) \right)}.$
Finally, using that $z \in \T^d \mapsto b^{-1} \cdot z$ is continuous and preserves the Hausdorff distance (Lemma \ref{lm:trans_haus}), this yields the first part of the Theorem.
\ \\
For the second statement of the Theorem, we apply the first part with a point $b = (b_1, \ldots, b_d) \in \T^d$ which is admissible for all the joint spectra $\mathrm{JointSpec}(U_1(\hbar), \ldots, U_d(\hbar))$, $\hbar \leq \hbar_0$, and such that the set $b \cdot F(M)$ is very simple, keeping in mind Definition \ref{def:CH_simple}.
\end{proof}}

\section{Application to symplectic geometry}
\label{sect:circle_BTO}

Symplectic actions that are not Hamiltonian have recently become of important relevance in view of the work of Susan Tolman~\cite{MR3735631} (which constructs many such actions with isolated fixed points on compact manifolds) and recent works studying when a symplectic action is Hamiltonian, and the closely related problem of estimating the number of fixed points of a symplectic non\--Hamiltonian action,  see for example \cite{LinPel,TolWei} and references therein. For an action which is symplectic but not Hamiltonian, there is no momentum map in the usual sense, but one can construct a circle-valued function playing the same part.

\subsection{Construction of the circle valued momentum map}

We identify  $\mathbb{S}^1$ with $\mathbb{R}/\mathbb{Z}$ and denote by $\pi:
\mathbb{R}\ni t\mapsto[t] \in \mathbb{R}/\mathbb{Z}$ the projection. The length form $\lambda \in
\Omega^1(\mathbb{R}/\mathbb{Z})$ is given by $\lambda([t])\left(T_t \pi(r)\right):=r$.  
Let $(M,\omega)$ be a connected symplectic manifold, that is, 
$M$ is a smooth manifold and $\omega$  is a smooth $2$\--form on $M$ which is  
non-degenerate and closed.  Let $\Phi: (\mathbb{R}/\mathbb{Z}) \times M 
\rightarrow M$ be a smooth symplectic action, that is a
smooth action by diffeomorphisms 
$\Phi_{[t]}: M \rightarrow M$ that preserves the 
symplectic form $\omega$ (these are called \emph{symplectomorphisms}). 
For $r \in \mathbb{R}$ denote by
$r_M \in \mathfrak{X}(M)$ the action infinitesimal generator 
given by $r_M(x): = \left.\frac{d}{d\varepsilon}\right|_{\varepsilon=0}\Phi_{[r\varepsilon]}(x).$

\begin{definition}
The  $\mathbb{R}/\mathbb{Z}$\--action on $(M,\omega)$  is 
\emph{Hamiltonian} if there is a smooth 
map $\mu \colon M \to \mathbb{R}$
such that $\mathbf{i}_{1_M}\omega:= 
\omega(1_M, \cdot) = {\rm d} \mu.$
The map $\mu$ is called the \emph{momentum map} of the action.
\end{definition}

Note that the existence of $\mu$ is equivalent to 
the one\--form 
$\mathbf{i}_{1_M}\omega$ being exact, and therefore
if the first cohomology group $H^1(M; \mathbb{R})$ vanishes
then every symplectic 
$\mathbb{R}/\mathbb{Z}$-action on $M$ is in fact Hamiltonian. 

If the $\mathbb{R}/\mathbb{Z}$\--action does not have a momentum map in the
sense above, then the action must be non trivial. Hence, if the action is not Hamiltonian,
then $\mathbf{i}_{1_M}\omega$ is not exact. These type of actions also admit
an analogue of the momentum map, called the \emph{circle valued momentum
map}, and which now takes values in $\mathbb{R}/\mathbb{Z}$. 
A \emph{circle valued momentum map} $\mu:M \rightarrow \mathbb{R}/\mathbb{Z}$ is
determined by  the equation 
$\mu^\ast \lambda = \mathbf{i}_{1_M} \omega.$
 
Such a map $\mu$ always exists, for either $\omega$ itself, or a very close perturbation
of it. To be more precise, suppose that $\mathbb{R}/\mathbb{Z}$ acts 
symplectically on the closed symplectic manifold $(M,\omega)$, but not Hamiltonianly. 
Whenever the symplectic form $\omega$ is integral (that is, 
$[\omega] \in H^2(M; \mathbb{Z})$), then the action admits a circle valued momentum map 
$\mu: M \rightarrow \mathbb{R}/\mathbb{Z}$ for $\omega$ (this result is due to McDuff, see \cite{McDuff},
and is valid for some symplectic form even when the integral cohomology assumption is invalid).

For the sake of completeness and because it is a very simple construction, we review it here.
It follows from \cite[Lemma 7]{PelRat} that 
$\left[\mathbf{i}_{1_M} \omega \right] \in H^1(M; \mathbb{Z}).$
Fix $m_0 \in M$ and let $\gamma_m$ be an arbitrary smooth path in $M$, from $m_0$ to $m$, and define $\mu:M \rightarrow \mathbb{R}/\mathbb{Z}$ by 
\begin{eqnarray} \label{mom} \mu(m):= \left[ \int_{\gamma _m}\mathbf{i}_{1_M} \omega\right].\end{eqnarray}
It is immediate that the definition of $\mu$ is independent of paths, so it is is well defined. Also, $\mu$ is clearly smooth, and for every $v _m\in T_m M$, we have 
$T_m\mu(v_m) = T_{\int_{\gamma _m} \mathbf{i}_{1_M} \omega} \pi \big(\mathbf{i}_{1_M} \omega(m)(v_m)  \big),$
and consequently 
$(\mu^\ast\lambda)(m)(v_m) = \lambda( \mu(m)) \left(T_m \mu(v_m)\right) = \left(\mathbf{i}_{1_M} \omega\right)(m)(v_m),$
as desired. The map $\mu$ is defined up to the addition of constants (due to the freedom in the choice of $m_0$).

\subsection{Circle action and Berezin-Toeplitz quantization}

It turns out that there exists a natural way to derive a semiclassical quantization of this circle-valued moment map when $M$ is compact and $\omega$ is integral (in fact, integral up to a factor $2\pi$); this semiclassical quantization is called Berezin-Toeplitz quantization. It builds on geometric quantization, due to Kostant \cite{Kos} and Souriau \cite{Sou}. Berezin-Toeplitz operators were introduced by Berezin \cite{Ber}, their microlocal analysis was initiated by Boutet de Monvel and Guillemin \cite{BouGui}, and they have been studied by many authors since (see for instance the review \cite{Schli} and the references therein).

Assume that $(M, \omega)$ is a compact, connected, K\"ahler manifold, which means that it is endowed with an almost complex structure which is compatible with $\omega$ and integrable. We recall that an almost complex structure $j$ on $M$ is a smooth section of the bundle $\mathrm{End}({T}M) \to M$ such that 
$j^2 = -\mathrm{id}_{{T}M}$, and $j$ being integrable means that it induces on $M$ a structure of complex manifold. Compatibility between $\omega$ and $j$ means that $\omega(\cdot,j \cdot)$ is a Riemannian metric on $M$.

Assume that the cohomology class $[\omega \slash 2 \pi]$ lies in $H^2(M,\mathbb{Z})$. Then there exists a prequantum line bundle $L \to M$, that is a holomorphic, Hermitian complex line bundle whose Chern connection (the unique connection compatible with both the holomorphic and Hermitian structures) has curvature form equal to $-i \omega$. Then for any integer $k \geq 1$, the space
\[ \Hil_k =  H^0\left(M,L^{\otimes k}\right) \]
of holomorphic sections of the line bundle $L^{\otimes k} \to M$, endowed with the Hermitian product
\[ \phi, \psi \in \Hil_k \mapsto \scal{\phi}{\psi}_k = \int_M h_k(\phi, \psi) \mu_M \]
where $\mu_M$ is the Liouville measure associated with $\omega$ and $h_k$ is the Hermitian form on $L^{\otimes k}$ inherited from the one of $L$, is a finite dimensional Hilbert space.

Now, the quantization map $\mathrm{Op}_k: \classe{\infty}{(M,\mathbb{C})} \to \mathcal{L}(\Hil_k)$
is defined as follows: let $L^2(M,L^{\otimes k})$ be the space of square integrable sections of the line bundle $L^{\otimes k} \to M$, that is the completion of $\classe{\infty}(M,L^{\otimes k})$ with respect to $\scal{\cdot}{\cdot}_k$, and let $\Pi_k$ be the orthogonal projector from $L^2(M,L^{\otimes k})$ to $\Hil_k$. Then, given $f \in \classe{\infty}{(M,\mathbb{C})}$, let $\mathrm{Op}_k(f) = \Pi_k f$
where, by a slight abuse of notation, $f$ stands for the operator of multiplication by $f$ in $L^2(M,L^{\otimes k})$. Here the integer parameter $k$ plays the part of the inverse of $\hbar$, therefore the semiclassical limit corresponds to $k \to + \infty$ instead of $\hbar \to 0$.

\begin{lemma}
\label{lm:BTO_axioms}
The Berezin-Toeplitz quantization is a semiclassical quantization.
\end{lemma}

\begin{proof}
This work was done in \cite{PelPolVu} for axioms \ref{item:normalization} to \ref{item:product}. The fact that axiom \ref{item:composition} is satisfied comes, for instance, from \cite[Section 5]{BordMeinSchli}. Let us show that axiom \ref{item:reality} holds. For $\phi, \psi \in \Hil_k$, we have that
\[ \scal{\Pi_k (f \phi)}{\psi}_k = \scal{f \phi}{\psi}_k =  \int_M h_k(f\phi,\psi) \ \mu_M \]
because $\Pi_k$ is self-adjoint and $\Pi_k \psi = \psi$; by sesquilinearity of $h_k$, this yields
\[ \scal{\phi}{\bar{f} \psi}_k = \int_M h_k(\phi,\bar{f} \psi) \ \mu_M = \scal{\phi}{\bar{f} \psi}_k =  \scal{\phi}{\Pi_k (\bar{f} \psi)}_k.\]
This means that $\mathrm{Op}(f)^* = \mathrm{Op}(\bar{f})$.
\end{proof}

\begin{remark}
We have assumed that $M$ is K\"ahler for convenience, but there exist ways to construct a Berezin-Toeplitz quantization on a compact symplectic, not necessarily K\"ahler, manifold $(M, \omega)$ with $[\omega/(2\pi)]$ integral, see for instance \cite{BorUri,MaMa,Cha_symp}.
\end{remark}

Assume now that $M$ is endowed with a smooth symplectic, but not Hamiltonian, action of $\mathbb{S}^1$. We now identify $\R \slash \Z$ with the unit circle $\mathbb{S}^1$ in $\mathbb{C}$ by means of the map
$ \R \slash \Z \to \mathbb{S}^1, [t] \mapsto \exp(2i\pi t).$
Since the symplectic form $\tilde{\omega} = \omega / 2 \pi$ is integral, there exists a circle valued momentum map $\tilde{\mu}$ with respect to $\tilde{\omega}$ for the action, whose value at $m \in M$ is given by the formula
$ \tilde{\mu}(m)= \left[ \int_{\gamma _m}\mathbf{i}_{1_M} \tilde{\omega} \right],$
where $\gamma_m$ is a smooth path connecting a fixed point $m_0 \in M$ to $m$. Hence we get a function $\mu \in \classe{\infty}{(M,\mathbb{S}^1)}$ defined as 
$\mu(m) = \exp(2i\pi \tilde{\mu}(m)) = \exp \left(i \int_{\gamma _m}\mathbf{i}_{1_M} \omega \right).$
We associate to this function a unitary Berezin-Toeplitz operator as follows. Set $V(k) = \mathrm{Op}_k(\mu)$; then $V(k)$ is a Berezin-Toeplitz operator with principal symbol $\mu$ but may not be unitary. However, the operator 
$U(k) := V(k) \left(V(k)^* V(k) \right)^{-1/2}$
is well-defined, clearly unitary, and it follows from the stability of Berezin-Toeplitz operators with respect to smooth functional calculus \cite[Proposition 12]{Cha} that it is a Berezin-Toeplitz operator with principal symbol $\mu$.

\subsection{A family of examples}
\label{sect:example}

Following these constructions, we introduce a family of examples for manifolds $M = \T^{2d}$. We start with the case $d=1$.\\

\paragraph{\textbf{An example when $d=1$}}

A famous example of symplectic but non Hamiltonian circle action is the action of $\mathbb{S}^1 = \R \slash \mathbb{Z}$ on $\mathbb{T}^2 = \R^2 \slash \mathbb{Z}^2$ given by the formula: 
$[t] \cdot \left([q,p] \right) = \left([t + q, p] \right).$
Here the torus $\T^2 = \R^2 \slash \mathbb{Z}^2$ is endowed with the symplectic form coming from the standard one on $\R^2$, that is:
$\omega =  dp \wedge dq.$
The action is clearly symplectic, and is not Hamiltonian, for instance because it has no fixed point.
\begin{lemma}
The circle-valued momentum map associated with this action is 
$\tilde{\mu}([q,p]) = [p]$ up to the addition of a constant.
\end{lemma}

\begin{proof}
Using the notation of the previous section, we have that 
$\Phi_{[t]}([q,p]) = [t + q,p],$
hence
$1_M([q,p]) =  \frac{\partial }{\partial q},$
therefore $\mathbf{i}_{1_M} \omega = dp$. Take $m_0 = [0,0] \in \T^2$ and let $m = [q,p]$ be any point in $\T^2$. Then
$ \gamma_m: [0,1] \to \T^2, \quad t \mapsto [tq,tp]$
is a smooth path connecting $m_0$ to $m$. Thus
$\tilde{\mu}(m) = \int_{\gamma_m} dp = \int_0^1 p \ dt = p.$
\end{proof}

As in the previous part, this map gives rise to a map 
$\mu \in \classe{\infty}{\left(\T^2,\mathbb{S}^1 \right)}$, $\mu([q,p]) = \exp(2i \pi p).$
We have a natural semiclassical operator associated with this momentum map, in the setting of Berezin-Toeplitz quantization. Firstly, let us briefly describe the geometric quantization of the torus, although it is now quite standard (see \cite[Chapter I.3]{Mum} for instance). Let $L_{\R^2} \to \R^2$
be the trivial line bundle with standard Hermitian form and connection $d-i\alpha$, where $\alpha$ is the 1-form defined as $\alpha_u(v) = \frac{1}{2}\omega(u,v),$
equipped with the unique holomorphic structure compatible with the Hermitian structure and the connection. Consider a lattice $\Lambda \subset \R^2$ of symplectic volume $4 \pi$. The Heisenberg group 
$H = \R^2 \times U(1)$ with product
$ (x,u) \star (y,v) = (x+y,uv\exp(\frac{i}{2}\omega_{0}(x,y)))$
acts on $L_{\R^2}$, with action given by the same formula. This action preserves all the relevant structures, and the lattice $\Lambda$ injects into $H$; therefore, by taking the quotient, we obtain a prequantum line bundle $L$ over $\mathbb{T}^2 = \R^2 \slash \Lambda$. Furthermore, the action extends to the line bundle $L_{\R^2}^{\otimes k}$ by
$(x,u).(y,v) = (x+y,u^k v\exp(\frac{ik}{2}\omega_{0}(x,y))).$
We thus get an action $T^*:  \Lambda  \rightarrow  \text{End}(\classe{\infty}{(\R^2,L_{\R^2}^{\otimes k})}), \quad u  \mapsto  T_{u}^*.$ The Hilbert space $\Hil_{k} = H^0(\T^2,L^{\otimes k})$ can naturally be identified with the space $\Hil_{\Lambda,k}$ of holomorphic sections of $L_{\R^2}^{\otimes k}  \rightarrow \R^2$ which are invariant under the action of $\Lambda$, endowed with the Hermitian product
$\langle \phi,\psi \rangle_k = \int_{D}\phi \overline{\psi} \ |\omega|$
where $D$ is the fundamental domain of the lattice. Furthermore, $\Lambda \slash 2k$ acts on $\Hil_{\Lambda,k}$. Let $e$ and $f$ be generators of $\Lambda$ satisfying $\omega(e,f) = 4 \pi$; one can show that there exists an orthonormal basis $(\psi_{\ell})_{\ell \in \Z \slash 2k\Z}$ of $\Hil_{\Lambda,k}$ such that
\[ \forall \ell \in \Z \slash 2k\Z \qquad \left\{\begin{array}{c} T^*_{e/2k} \psi_{\ell} = w^{\ell} \psi_{\ell}  \\ \\ T^*_{f/2k} \psi_{\ell} = \psi_{\ell + 1} \end{array}\right. \]
with $w = \exp\left( \frac{i \pi}{k} \right)$. The $\psi_{\ell}$ can be computed using Theta functions.

Now, set  $U(k) = T^*_{e/2k}: \Hil_k \to \Hil_k;$ 
of course, $U(k)$ is unitary. Let $(q,p)$ be coordinates on $\R^2$ associated with the basis $(e,f)$ and $\left[q,p\right]$ be the equivalence class of $(q,p)$. It is known \cite[Theorem $3.1$]{ChaMar} that $U(k)$ is a Berezin-Toeplitz operator with principal symbol $\left[q,p\right]  \mapsto  \exp(2 i \pi p),$
which is precisely $\mu$. Trivially, $\mathrm{Sp}(U(k)) = \left\{ \exp(i \pi \ell/k),  0 \leq \ell \leq 2k-1 \right\}$
which is dense in $\mu(\T^2) = \mathbb{S}^1$ when $k$ goes to infinity. Thus, this example is interesting because the assumptions of Theorem \ref{thm:CH_simpleF} are not satisfied, since $\mu$ is onto, yet we can recover $\mu(M)$ from the spectrum of $U(k)$ when $k \to + \infty$.  \\

\paragraph{\textbf{The higher dimensional case.}}

More generally, we can consider $d$ symplectic but non Hamiltonian circle actions on $M = \mathbb{T}^{2d} = (\T^2)^d$, endowed with the symplectic form coming from
$\omega = {\rm d}p_1 \wedge {\rm d}q_1 + \ldots + {\rm d}p_d \wedge {\rm d}q_d$
as follows: for $j \in \llbracket 1, d \rrbracket$, the $j$-th action is the action of $\mathbb{S}^1$ described above applied to the $j$-th copy of $\T^2$:
\[ [t].[q_1, p_1, \ldots,q_{d},p_d] = [q_1, p_1, \ldots, q_{j-1}, p_{j-1}, t + q_j, p_j, q_{j+1}, p_{j+1}, \ldots,q_{d},p_d].  \]
This action admits the circle valued moment map $\mu_j \in \classe{\infty}{\left(\T^{2d},\mathbb{S}^1\right)},$
where $\mu_j([q_1, p_1, \ldots,q_{d},p_d]) = \exp(2i \pi p_j).$ Now, we recall the following useful property of Berezin-Toeplitz quantization with respect to direct products: if $M_1, M_2$ are two compact connected K\"ahler manifolds endowed with prequantum line bundles $L_1$ and $L_2$ respectively, the line bundle
\[ L = L_1 \boxtimes L_2 := \pi_1^*L_1 \otimes \pi_2^* L_2 \to M = M_1 \times M_2\]
is a prequantum line bundle (here $\pi_j: M \to M_j$ is the natural projection). Moreover, the quantum Hilbert spaces satisfy
\[ H^0(M,L^{\otimes k}) = H^0(M_1,L_1^{\otimes k}) \otimes H^0(M_2,L_2^{\otimes k}) \]
and, if $f_j \in \classe{\infty}{(M,\mathbb{C})}$, $j = 1,2$, then 
$\mathrm{Op}_k(f) = \mathrm{Op}_k(f_1) \otimes \mathrm{Op}_k(f_2)$
for $f(m_1,m_2) = f(m_1)f(m_2)$. Coming back to our example where the manifold is $M = \T^2 \times \ldots \times \T^2$, we quantize $\T^2$ as explained in the previous section and we obtain a family of quantum spaces
$\Hil_k =  H^0(\T^{2}, L^{\otimes k})^{\otimes d}$
with orthonormal basis 
$(\psi_{\ell_1} \otimes \ldots \otimes \psi_{\ell_d})_{\ell_1, \ldots, \ell_d \in \Z \slash 2k\Z}.$
Let $U(k)$ be the same operator as in the previous section, and introduce the operator
\[ V_j(k) := \mathrm{Id} \otimes \ldots \otimes \mathrm{Id} \otimes \underbrace{U(k)}_{j-\mathrm{th \ position}} \otimes \ldots \otimes \mathrm{Id}  \]
for every $j \in \llbracket 1,d \rrbracket$. Then $(V_1(k), \ldots, V_d(k))$ is a family of pairwise commuting unitary Berezin-Toeplitz operator acting on $\Hil_k$, with joint principal symbol $\mu = (\mu_1, \ldots, \mu_d)$. Its joint spectrum is equal to
\[  \left\{ \left( \exp\left( \frac{i \pi \ell_1}{k} \right), \ldots, \exp\left( \frac{i \pi \ell_d}{k} \right) \right), \ \ell_1, \ldots, \ell_d \in \Z \slash 2k\Z \right\} \]
and again, from this we recover $\mu(M) = \T^d$ when $k$ goes to infinity.

 \section{Final remarks}

We conclude with some remarks.

\begin{enumerate}
\item
The results of this paper do not directly follow from the self\--adjoint case for two reasons. First, in order to use the result obtained in \cite{PelPolVu} for self\--adjoint operators, which seems to be a fairly natural plan of attack, we wanted to transform our unitary operators into self-adjoint operators, using the Cayley transform, which can only be applied to unitary operators not containing $-1$ in their spectrum. Second, both the joint spectrum of a family of commuting semiclassical unitary operators and the image of its joint momentum map are subsets of a $d$\--torus and dealing with convex hulls inside is not obvious;  the naive notion of convex hull (obtained by lifting, taking the convex hull and then projecting back) leads to a set which is in general much larger that the actual set; henceforth one of our goals in the second appendix will be to give a procedure to  find the convex hull which leads to the desired convergence result. 
\item
It would of course be more satisfying to prove that the quantum joint spectrum converges to the classical spectrum, without mentioning convex hulls. This problem has been investigated by Pelayo and V{\~u} Ng{\d{o}}c \cite{PelVu} in the context of self-adjoint operators. Nevertheless, even in the self\--adjoint case, getting rid of these convex hulls does not come for free; one needs to introduce an additional axiom, which restricts the class of operators to which the result can be applied. Indeed, as explained in the article cited above, this axiom is not satisfied by general classes of pseudodifferential operators, but only by particular classes, such as the one of pseudodifferential operators with uniformly bounded symbols. Consequently, in order to state a result which is as general as possible, we have not tried to get rid of convex hulls here, even though this would have greatly simplified this particular aspect of the problem.
\item
We would like to get rid of the assumption on the surjectivity of the principals symbols. We consider pairwise commuting unitary semiclassical operators 
$U_1(\hbar), \ldots, U_d(\hbar)$ with joint principal symbol 
$F = (f_0^1, \ldots, f_0^d).$
We still assume that $F(M)$ is closed. We conjecture the following: assume that $F(M)$ is hullizable (see Definition \ref{def:hullizable}). Then from the behaviour of the joint spectrum $\mathrm{JointSpec}(U_1(\hbar), \ldots, U_d(\hbar))$ when $\hbar$ goes to zero, one can recover the convex hull of $F(M)$. 
We have given evidence for this conjecture in Section \ref{sect:example}, but first let us make a few comments about it. Firstly, ``recover'' can have several meanings, but it would be appreciable to obtain a statement similar to Theorem \ref{thm:CH_simpleF} involving the convex hull of the joint spectrum; however, the latter may no longer be simple, so we would need to give a meaning to its convex hull. Secondly, in order to prove this conjecture, using axioms \ref{item:composition} to  \ref{item:func_calc} only might not be enough, thus a natural problem would be to look for the minimal set of additional axioms needed for this proof.
\item
The interest for non self\--adjoint operators in physics in general and quantum mechanics in particular has been growing in the last few decades. They appear, for instance, in the study of damping and resonances, see \cite{Dav} for a review. Moreover, much attention has been given to $\mathcal{PT}$-symmetric operators (see \cite{Ben} for an introduction), which are non self\--adjoint but may still have a real spectrum under certain conditions. Let us also mention that semiclassical techniques have been applied to describe the structure of the spectrum of certain non self\--adjoint perturbations of self\--adjoint operators \cite{HitSjo, Rou}. For all these reasons, including non self\--adjoint operators in an axiomatic definition of semiclassical operators is relevant. \end{enumerate}

\section*{Appendix 1: Basic operator theory}

Let $\Hil$ be a Hilbert space, with scalar product $\scal{\cdot}{\cdot}$; we use the notation $\|\cdot \|$ for the associated norm. We will need to work with possibly unbounded linear operators acting on $\Hil$, hence we introduce some standard terminology (for more details, we refer the reader to standard material, as \cite[Chapter VIII]{ReedSimon} or \cite[Appendix 3]{HislopSigal} for instance). A linear operator acting on $\Hil$ is the data of a linear subspace $\mathcal{D}(T) \subset \Hil$, called the domain of $T$, and a linear map 
$T: \mathcal{D}(T) \to \Hil$. Throughout the paper, $\mathcal{L}(\Hil)$ will denote the set of densely defined (that is with dense domain) linear operators on $\Hil$. The range $\mathcal{R}(T)$ of a linear operator $T$ is the set of all values $Tu$, $u \in \mathcal{D}(T)$.

We say that the operator $T$ is \emph{bounded} if there exists a constant $C \geq 0$ such that for every $u \in \mathcal{D}(T)$, $\| Tu\| \leq C \|u\|.$ If this is the case, by a slight abuse of notation, we will write $\|T\|$ for its operator norm, defined as 
\[ \|T\| = \sup_{\substack{u \in \mathcal{D}(T), u \neq 0}} \frac{\|Tu\|}{\|u\|}.\]
Let us recall that if $T$ is a bounded operator, it admits a bounded extension with domain $\Hil$ (see \cite[Proposition A.3.9]{HislopSigal} for example). 

If $T$ is a densely defined linear operator acting on $\Hil$, its adjoint is defined as follows: let $\mathcal{D}(T^*)$ be the set of $u \in \Hil$ such that there exists $v_u \in \Hil$ satisfying: $\forall w \in \mathcal{D}(T), \scal{Tw}{u} = \scal{w}{v_u}$. Then for $u \in \mathcal{D}(T^*)$, this $v_u$ is unique and we set $T^* u = v_u$. This defines a linear operator acting on $\Hil$, with domain $\mathcal{D}(T^*)$ not necessarily dense; $T^*$ is called the adjoint of $T$. A densely defined closed operator is said to be \emph{normal} when $T T^* = T^* T$ (this equality includes the fact that the domains of these operators agree). Normal operators are of particular interest because they satisfy the spectral theorem \cite[Chapter X, Theorem 4.11]{Conway} which associates to the operator a spectral measure and spectral projections. Two normal operators $A, B \in \mathcal{L}(\Hil)$ are said to \emph{commute} if and only if all their spectral projections commute (cf. for instance \cite[Proposition 5.27]{Schmu}). A densely defined operator $T$ is said to be \emph{self-adjoint} when $T^* = T$.

An operator $T \in \mathcal{L}(\Hil)$ is said to be \emph{positive}, in which case we will write $T \geq 0$, when $\scal{T u}{u} \geq 0$ for every $u \in \mathcal{D}(T)$; if there exists some constant $c \in \R$ such that $T - c \ \mathrm{Id} \geq 0$, then we write $T \geq c \ \mathrm{Id} $.

We say that $T \in \mathcal{L}(\Hil)$ is \emph{invertible} if it admits a bounded inverse, that is a bounded operator 
$T^{-1}:\mathcal{R}(T) \to \mathcal{D}(T)$
such that $T T^{-1} = \mathrm{Id}_{\mathcal{R}(T)}$ and $T^{-1} T = \mathrm{Id}_{\mathcal{D}(T)}$. In this case, $T^{-1}$ is unique. A bounded operator $U$ acting on $\Hil$ is said to be unitary if it is invertible and $U^{-1} = U^*$. Now, we define the spectrum $\mathrm{Sp}(T) \subset \mathbb{C}$ of a given $T \in \mathcal{L}(\Hil)$ as follows: 
$\lambda \in \mathrm{Sp}(T)$ if and only if $\lambda \mathrm{Id} - T$ is \ not \ invertible.
It is standard that the spectrum of a self-adjoint (respectively unitary) operator is a subset of $\R$ (respectively the unit circle $\mathbb{S}^1$).

Finally, recall the following useful result about the norm of a self-adjoint operator. If $A$ is self-adjoint, then
\begin{equation} \sup_{\lambda \in \mathrm{Sp}(A)} |\lambda| = \sup_{\substack{u \in \mathcal{D}(A) \\ u \neq 0}} \frac{|\scal{Au}{u}|}{\|u\|^2} = \sup_{\substack{u \in \mathcal{D}(A) \\ u \neq 0}} \frac{\| Au\|}{\|u\|} \leq + \infty. \label{eq:rayleigh}\end{equation}
This result is standard but very often stated for bounded operators only; a concise proof can be found in \cite[Section 3]{PelPolVu}.

 A finite number of normal operators $S_1,\ldots,S_d$ on a Hilbert space are said to be \emph{mutually commuting}
if  their corresponding spectral measures $\mu_1,\dots,\mu_d$ pairwise commute. In this case we may define the joint spectral measure  $\mu := \mu_1\otimes\cdots\otimes\mu_d$ on $\mathbb{C}^d$.
 We are concerned with semiclassical operators, that is, the operator itself is given by a sequence of operators, labelled by the Planck constant $\hbar$, considered as a small parameter. Let $I$ be a subset of $(0,1]$ that accumulates at $0$. Let 
\[ \mathcal{F} =(T_1:=(T_1(\hbar))_{\hbar \in I},\dots,T_d:=(T_d(\hbar))_{\hbar \in I}) \]
be a collection of pairwise commuting semiclassical normal operators. These operators depend on the parameter $\hbar \in I$ and act on a Hilbert space $\mathcal{H}_{\hbar},\; \hbar\in I$.  We assume that at each $\hbar \in I$ the operators have a common dense domain $\mathcal{D}_{\hbar} \subset \mathcal{H}_{\hbar}$ such that the inclusion $T_j(\hbar)(\mathcal{D}_{\hbar}) \subset
\mathcal{D}_{\hbar}$ holds for all $j=1,\ldots, d$. For a fixed value of $\hbar$, the \emph{joint spectrum} of $(T_1(\hbar),\ldots,T_d(\hbar))$ is the support of their joint spectral measure. It is denoted by $\op{JointSpec}(T_1(\hbar),\ldots,T_d(\hbar))$.
For instance, if the Hilbert space $\mathcal{H}_\h$ is finite dimensional (which is the case for instance with Berezin-Toeplitz operators on closed symplectic manifolds), then $\op{JointSpec}(T_1(\hbar),\ldots,T_d(\hbar))$ is the set of $(\lambda_1,\ldots,\lambda_d) \in \R^d$ such that there exists $v\neq 0$ satisfying $T_j(\hbar) v = \lambda_j v$ for all $j=1,\ldots,d$.  
The \emph{joint spectrum} $\op{JointSpec}(T_1,\ldots,T_d)$ of $(T_{1},\ldots,T_{d})$ is the collection of all joint spectra of $T_{1}(\hbar),\ldots,T_{d}(\hbar)$ for $\hbar \in I$.

\section*{Appendix 2: Taking convex hull in tori}
\label{sect:convex_hull_torus}

In this section we propose a way to define the convex hull of a subset of the torus $\T^d$. We believe that the definition of toric convex hull that we introduce here could be of interest in other settings, for instance in computational geometry; see \cite{GriMar} where a first definition was proposed, but the convex hulls constructed using it were too large for practical applications.

\subsection*{Preliminaries about $\T^d$}

We consider $\T^d = (\mathbb{S}^1)^d$ as the product of $d$ copies of the unit circle. If $z$ belongs to the unit circle, we will denote by $\arg(z)$ its argument in $(-\pi,\pi]$. Furthermore, we will also denote by $\arg: \mathbb{T}^d = (\mathbb{S}^1)^d \to (-\pi,\pi]^d$
the function assigning its argument to each component of $z \in \mathbb{T}^d$:
$\arg(z_1,\ldots,z_d) = \left(\arg(z_1),\ldots,\arg(z_d) \right).$
Similarly, we will consider the function
\[\exp: \mathbb{C}^d \to \mathbb{C}^d, \quad (w_1, \ldots, w_d) \mapsto (\exp(w_1), \ldots, \exp(w_d)). \]

We endow $\mathbb{T}^d$ with the following distance: for $z, w \in \mathbb{T}^d$
\[ d^{\mathbb{T}^d}(z,w) = \min_{\theta \in (2\pi\Z)^d} \| \arg(z) - \arg(w) + \theta \|_{\R^d} . \]
The Hausdorff distance induced by this distance will be denoted by $d_{H}^{\T^d}$.

\subsection*{Multiplication in $\mathbb{T}^d$}

Let $a = (a_1,\ldots,a_d), b = (b_1,\ldots,b_d)$ be two points in $\mathbb{T}^d = (\mathbb{S}^1)^d$. Then we use the following notation for the product of $a$ and $b$ in $\mathbb{T}^d$: 
$a \cdot b = (a_1 b_1, \ldots, a_d b_d).$
Now, given a subset $E$ of the torus $\mathbb{T}^d$ and a point $a \in \mathbb{T}^d$, we define the set $a.E$ as the set of all points of the form $a \cdot z$, $z \in E$. Moreover, we use the notation $a^{-1} \in \mathbb{T}^d$ to denote the point $(a_1^{-1}, \ldots, a_d^{-1})$. An easy consequence of our choice of distance on $\T^d$ is that the associated Hausdorff distance $d_H^{\T^d}$ is multiplication invariant.

\begin{lemma}
\label{lm:trans_haus}
Let $E, F \subset \mathbb{T}^d$. Then 
$d_H^{\mathbb{T}^d}(a \cdot E, a \cdot F) = d_H^{\mathbb{T}^d}(E, F)$
for every $a \in \mathbb{T}^d$.
\end{lemma}

\subsection*{Convex hulls for simple subsets of $\mathbb{T}^d$}

If we could lift everything to $\mathbb{R}^d$ without any trouble, we would define the convex hull of a subset $E$ of $\T^d$ as the projection of the convex hull of its lift. This naive idea cannot be used in general, but can be adapted for what we will call simple subsets (see the definitions below), whose image under $\arg$ is contained in $(-\pi,\pi)^d$. For instance, it works as is if $E$ is simple and connected. However, if we drop connectedness, which is a natural thing to do since one often wants to compute the convex hull of a collection of points, then choices are involved, and defining the convex hull is subtle.

\begin{definition}
A subset $E \subset \mathbb{T}^d$ is called \emph{very simple} if for every point $(z_1,\ldots,z_d) \in E$ and for all $j \in \llbracket 1,d \rrbracket$, $z_j \neq -1$.
\end{definition}

For $j \in \llbracket 1,d \rrbracket$, let $p_j: \mathbb{T}^d = (\mathbb{S}^1)^d \to \mathbb{S}^1$ be the projection on the $j$-th factor.

\begin{definition}
A subset $E \subset \mathbb{T}^d$ is called \emph{simple} if no ${p_j}_{|E}$ is onto.
\end{definition}

\begin{remark}
If $E$ is simple, there exists $a \in \mathbb{T}^d$ such that $a \cdot E$ is very simple. A set consisting of a finite number of points is always simple.
\end{remark}

Recall that for a subset $E$ of the torus $\mathbb{T}^d$ and a point $a \in \mathbb{T}^d$, we define the set $a \cdot E$ as the set of points of the form $a \cdot z$, $z \in E$. 

\begin{lemma}
\label{lm:thetas}
Let $E \subset \mathbb{T}^d$ be simple and compact, with finitely many connected components $E_1, \ldots, E_n$. Let $b,c \in \T^d$ be such that both $b \cdot E$ and $c \cdot E$ are very simple. Then for every $j \in \llbracket 1, N \rrbracket$, there exists a constant $\theta_j^{(b,c)} \in (2\pi\Z)^d$ such that for all $z \in E_j$, 
\[ \arg(c \cdot z) = \arg(b \cdot z) + \arg(c \cdot b^{-1}) + \theta_j^{(b,c)}.\]
We call $\theta_j^{(b,c)}$ the \emph{phase shift} of $E_j$ with respect to $(b,c)$.
\end{lemma}

\begin{proof}
Let $z \in E$; then $c \cdot z = (c \cdot b^{-1}) \cdot (b \cdot z)$, therefore 
\[ \arg(c \cdot z) = \arg(b \cdot z) + \arg(c \cdot b^{-1}) + \theta(z) \]
for some $\theta(z) \in (2\pi \Z)^d$. But the function $z \mapsto \arg(c \cdot z) - \arg(b \cdot z)$ is continuous, since $b \cdot E$ and $c \cdot E$ are very simple; indeed, if a compact set $H \subset \T^d$ is very simple, then $H$ is contained in some compact subset $K$ of $(\mathbb{S}^1 \setminus \{-1\})^d$, and the function $\arg: K \to (-\pi,\pi)^d$ is continuous. Hence the same holds for $z \mapsto \theta(z)$, and thus $\theta(z) = \theta(w)$ whenever $z$ and $w$ belong to the same connected component of $E$. Consequently, for every $j \in \llbracket 1,d \rrbracket$, there exists a constant $\theta_j^{(b,c)} \in (2\pi\Z)^d$ such that for every $z \in E_j$, $\theta(z) = \theta_j^{(b,c)}$, which was to be proved.
\end{proof}

Given a compact subset $I$ of $\R$, we use the notation $\mathrm{diam}(I)$ for the diameter of $I$:
$\mathrm{diam}(I) = \max \left\{ |x-y|, \quad x,y \in I \right\},$
with the convention that $\mathrm{diam}(\emptyset) = 0$. Furthermore, let $\eta_j$, $1 \leq j \leq d$ denote the natural projection $\R^d \to \R, \ (x_1, \ldots, x_d) \mapsto x_j$.

\begin{lemma}
\label{lm:diam}
Let $E \subset \mathbb{T}^d$ be simple and compact, with finitely many connected components. Then there exists $b \in \mathbb{T}^d$ such that $b \cdot E$ is very simple and 
\[ \forall j \in \llbracket 1,d \rrbracket, \quad \mathrm{diam}(\eta_j\left(\arg(b \cdot E)\right)) = \min \Lambda_j  \]
where 
\[ \Lambda_j = \left\{ \mathrm{diam}(\eta_j\left(\arg(c \cdot E)\right)) \left| \ c \in \T^d, c \cdot E \ \mathrm{very \ simple}  \right\}\right. .\] 
We say that such a point $b \in \T^d$ is \emph{admissible for $E$}. 
\end{lemma}

\begin{proof}
We start by proving that the sets $\Lambda_j$ do admit minima. Let $j \in \llbracket 1,d \rrbracket$; since $E$ is simple, $\Lambda_j$ is not empty. We will prove that the set $\Lambda_j$ consists of a finite number of values, which will yield the existence of its minimum.  We now make the following simple observation: if $G = \arg(c \cdot E) \subset (-\pi,\pi]^d$ with $c \cdot E$ very simple is the image of $F = \arg(b \cdot E) \subset (-\pi,\pi]^d$, with $b \cdot E$ very simple, by a translation, then 
$\mathrm{diam}(\eta_j(F)) = \mathrm{diam}(\eta_j(G)).$

Let $\theta_1^{(b,c)}, \ldots \theta_N^{(b,c)}$ be the phase shifts of $E_1, \ldots, E_N$ with respect to $(b,c)$ as introduced in Lemma \ref{lm:thetas}. If $G$ is not the image of $F$ by a translation, then necessarily there exists $i \neq k \in \llbracket 1, N \rrbracket$ such that $\theta_i^{(b,c)} \neq \theta_k^{(b,c)}$. But each $\theta_i^{(b,c)}$ is an element of $(2\pi \Z)^d \cap [-2\pi, 2\pi]^d$, hence we can only get a finite number of different values for $\theta_i^{(b,c)}$ by changing $b$ and $c$. Consequently, there is only a finite number of ways to make $G$ not be the image of $F$ by a translation, hence $\Lambda_j$ is finite.

It remains to prove that there exists a common $b \in \T^d$ minimizing all the $\Lambda_j$, $1 \leq j \leq d$. If $d =1$, this is obvious, thus let us assume that $d \geq 2$. Obviously we can pick some $b \in \T^d$ which is a minimizer for $\Lambda_1$. Now let $j \in \llbracket 1, d-1 \rrbracket$ and assume that we have found $b^j \in \T^d$ such that
\[ \forall i \in \llbracket 1,j \rrbracket, \quad \mathrm{diam}(\eta_i\left(\arg(b^j \cdot E)\right)) = \min \Lambda_i. \]  
Consider the set $C_j \subset \T^d$ of points of the form $(1,\tilde{c}) \cdot b_j$, $\tilde{c} \in \T^{d-j}$. Then clearly, for every $c \in C_j$ with $c \cdot E$ very simple,
\[ \forall i \in \llbracket 1,j \rrbracket, \quad \mathrm{diam}(\eta_i\left(\arg(c \cdot E)\right)) = \min \Lambda_i. \]
Now, let $\Xi_{j+1} = \left\{ \mathrm{diam}(\eta_{j+1}\left(\arg(c \cdot E))\right),c \in C_j, c \cdot E \ \mathrm{very \ simple}  \right\}$. We want to prove that $\min \Xi_{j+1} = \min \Lambda_{j+1}$; this follows from the fact that the map 
\[ \varphi_j: \T^j \times C_j \to \T^d, \quad (a,(1,\tilde{c}) \cdot b_j) \mapsto (a,\tilde{c}) \cdot b_j \]
is a bijection satisfying $\eta_{j+1} (\arg(\varphi_j(a,c) \cdot z)) = \eta_{j+1}(\arg(c \cdot z))$ for every $(a,c) \in \T^j \times C_j$ and $z \in E$. We conclude by (finite) induction.
\end{proof}

We would now like to define the convex hull of a subset $E \subset \T^d$ which is simple, compact, 
and has finitely many connected components, as the set
\[ b^{-1} \cdot \exp \left(i \mathrm{Convex \ Hull}(\arg(b \cdot E)) \right) \]
where $b \in \mathbb{T}^d$ is given by Lemma \ref{lm:diam} and where for any set $F \subset \R^d$, we define $\exp \left(i F \right) = \left\{\exp(i\theta), \ \theta \in F \right\}.$
The problem is that this point $b \in \mathbb{T}^d$ is, in general, far from being unique. Hence, in order to use this definition, we would need this set to not depend on the choice of $b$. But it can depend on this choice if $E$ displays some symmetries; this is why we use the following definition.

\begin{definition}
\label{def:CH_simple}
Let $E \subset \mathbb{T}^d$ be simple, compact, and with finitely many connected components $E_1, \ldots, E_N$. \begin{enumerate}
\item if for every $b, c \in \T^d$ which are admissible for $E$, all the phase shifts $\theta_1^{(b,c)}, \ldots, \theta_N^{(b,c)}$ are equal, then 
\[ \mathrm{Convex \ Hull}_{\mathbb{T}^d}(E) := b^{-1} \cdot \exp \left(i \ \mathrm{Convex \ Hull}(\arg(b \cdot E)) \right)\]
for any $b \in \T^d$ which is admissible for $E$,
\item otherwise, $\mathrm{Convex \ Hull}_{\mathbb{T}^d}(E) := \T^d$.

\end{enumerate}
\end{definition}

\begin{remark}
This definition implies that if $E \subset \T^d$ is simple, compact and connected, its convex hull is simply defined as 
\[ \mathrm{Convex \ Hull}_{\mathbb{T}^d}(E) := b^{-1} \cdot \exp \left(i \ \mathrm{Convex \ Hull}(\arg(b \cdot E)) \right) \]
for any $b \in \T^d$ such that $b \cdot E$ is very simple. 
\end{remark}

Definition \ref{def:CH_simple} makes sense because of the following lemma.

\begin{lemma}
\label{lm:admissible}
Let $b,c \in \T^d$ be two admissible points such that the equality $\theta_1^{(b,c)} =  \ldots = \theta_N^{(b,c)}$ holds. Then 
\[ c^{-1} \cdot \exp \left(i \ \mathrm{Convex \ Hull}(\arg(c \cdot E)) \right) = b^{-1} \cdot \exp \left(i \ \mathrm{Convex \ Hull}(\arg(b \cdot E)) \right). \]
\end{lemma}

\begin{proof}
Because of the assumption, we have that for every $z \in E$, the equality $\arg(c \cdot z) =  \arg(b \cdot z) + \arg(c \cdot b^{-1}) + \theta$ holds,
where $\theta$ is the common value of the $\theta_j^{(b,c)}$. Hence 
\[ \mathrm{Convex \ Hull}(\arg(c \cdot E)) = \arg(c \cdot b^{-1}) + \theta + \mathrm{Convex \ Hull}(\arg(b \cdot E)) \]
which implies, since $\theta$ belongs to $(2 \pi \Z)^d$, that
\[ \exp \left(i \ \mathrm{Convex \ Hull}(\arg(c \cdot E)) \right) = c \cdot b^{-1} \cdot \exp \left(i \ \mathrm{Convex \ Hull}(\arg(b \cdot E)) \right) \]
and the result follows.
\end{proof}

\begin{definition}
A simple compact subset $E \subset \T^d$ with finitely many connected components and satisfying the first condition in the above definition will be called \emph{generic}. 
\end{definition}

This terminology makes sense because such sets are, indeed, generic in the following sense.

\begin{lemma}
\label{lm:generic}
Let $E \subset \mathbb{T}^d$ be simple, compact, with finitely many connected components, and such that there exists $b, c \in \T^d$ admissible such that not all the phase shifts $\theta_j^{(b,c)}$ are equal. Then there exists $\varepsilon_0 > 0$ such that for every $\varepsilon \leq \varepsilon_0$, there exists a compact simple subset $E_{\varepsilon} \subset \T^d$ such that $E_{\varepsilon}$ is generic and $d_H^{\T^d}(E,E_{\varepsilon}) \leq \varepsilon$.
\end{lemma}

This result, of which we will give a proof later, is a corollary of the next two lemmas. The first one gives a necessary condition for a set to be non generic.

\begin{lemma}
\label{lm:bad_sets}
Let $E = \{ z_1, \ldots, z_N \} \subset \mathbb{T}^d$ be such that there exists $b, c \in \T^d$ admissible such that not all the phase shifts $\theta_j^{(b,c)}$ are equal. Then 
\begin{enumerate}
\item either there exist $p,q,r,s \in \llbracket 1, N \rrbracket$, $j \in \llbracket 1,d \rrbracket$ and $\theta \in 2 \pi \Z$ such that $\{p,q\} \neq \{r,s\}$ and 
\[ \left| \eta_j\left(\arg(b \cdot z_r) \right) - \eta_j\left(\arg(b \cdot z_s)\right) + \theta \right| = \left| \eta_j\left(\arg(b \cdot z_p)\right) - \eta_j\left(\arg(b \cdot z_q)\right) \right|,  \]
\item or there exist $p,q \in \llbracket 1, N \rrbracket$ and $j \in \llbracket 1,d \rrbracket$ such that 
\[ \left| \eta_j\left(\arg(b \cdot z_p)\right) - \eta_j\left(\arg(b \cdot z_q)\right) \right| = \pi. \]
\end{enumerate}
\end{lemma}

It would be interesting to give a simpler characterization of non generic sets. An example of non generic set when $d = 1$ is $E$ consisting of a finite number of points uniformly distributed on $\mathbb{S}^1$, but there are also sets with weaker symmetries which are not generic, for example 
\[ E = \{ \exp(i \phi_1), \exp(i \phi_2), \exp(i \phi_3) \} \subset \mathbb{S}^1 \]
where $\phi_3 = \pi + (\phi_1 + \phi_2)/2$ (see Figure \ref{fig:nongen}).

\begin{figure}[h]
\subfigure[$\left\{ \exp\left(\frac{i \ell \pi}{2}\right) \right\}, 0 \leq \ell \leq 3$]{\includegraphics{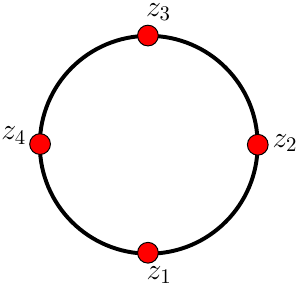}}
\hspace{1cm}
\subfigure[$ \left\{ \exp\left(\frac{-i \pi}{2}\right), 1, \exp\left(\frac{3i \pi}{4}\right) \right\}$ ]{\includegraphics{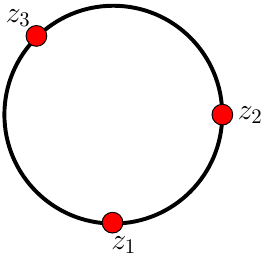}}
\caption{Two examples of non generic subsets of $\mathbb{S}^1$.}
\label{fig:nongen}
\end{figure}

\begin{proof}
Firstly, note that the existence of the pair $(b,c)$ satisfying the assumptions of the lemma implies that $N > 1$. Moreover, replacing $E$ by $b \cdot E$ and $c$ by $c \cdot b^{-1}$ if necessary, we can assume that $b=1$. To simplify the notation, we will set $\theta_{\ell} := \theta_{\ell}^{(1,c)}$, $1 \leq \ell \leq N$.\\

Let us start with some considerations for fixed $j \in \llbracket 1,d \rrbracket$. Let $p,q \in \llbracket 1,N \rrbracket$ be such that 
$\mathrm{diam}\left(\eta_j(\arg(E))\right) = \left| \eta_j(\arg(z_p)) - \eta_j(\arg(z_q)) \right|.$
If there exist $r,s \in \llbracket 1,N \rrbracket$ with $\{p,q\} \neq \{r,s\}$ such that this diameter is also equal to $\left| \eta_j(\arg(z_r)) - \eta_j(\arg(z_s)) \right|$, then we are done, because the property $(1)$ in the statement is satisfied with $\theta = 0$. So from now on we assume that it is not the case. We choose indices $r,s \in \llbracket 1,N \rrbracket$ such that 
\[ \mathrm{diam}\left(\eta_j(\arg(c \cdot E))\right) = \left| \eta_j(\arg(c \cdot z_r)) - \eta_j(\arg(c \cdot z_s)) \right|. \]
Since $1$ and $c$ are admissible, the equality
\[ \left| \eta_j(\arg(z_p)) - \eta_j(\arg(z_q)) \right| = \left| \eta_j(\arg(c \cdot z_r)) - \eta_j(\arg(c \cdot z_s)) \right| \]
holds; it can be rewritten as
\[ \left| \eta_j(\arg(z_p)) - \eta_j(\arg(z_q)) \right| = \left| \eta_j(\arg(z_r)) - \eta_j(\arg(z_s)) + \eta_j(\theta_r) - \eta_j(\theta_s) \right|.\]
If $\{p,q\} \neq \{r,s\}$, then we are done again, because property $(1)$ holds with $\theta = \eta_j(\theta_r) - \eta_j(\theta_s)$. If $\{p,q\} = \{r,s\}$ and $\eta_j(\theta_p) \neq \eta_j(\theta_q)$, then we are also done. Indeed, this means that 
\[ \left| \eta_j(\arg(z_p)) - \eta_j(\arg(z_q)) \right| = \left| \eta_j(\arg(z_p)) - \eta_j(\arg(z_q)) + \theta \right|,\]
where $\theta = \pm 2\pi$. Assuming for instance that $\eta_j(\arg(z_p)) > \eta_j(\arg(z_q))$, this yields 
$2(\eta_j(\arg(z_p)) - \eta_j(\arg(z_q))) = \pm 2 \pi.$

Therefore, let us consider the case where $\{p,q\} = \{r,s\}$ and $\eta_j(\theta_p) = \eta_j(\theta_q) = \mu;$ we will call this case the \emph{exceptional case}. Exchanging the roles of $p$ and $q$ if necessary, we can assume that $\eta_j(\arg(z_p)) > \eta_j(\arg(z_q))$. Then for every $\ell \notin \{p,q\}$, we have that 
\begin{equation} \eta_j(\arg(z_q)) < \eta_j(\arg(z_{\ell})) < \eta_j(\arg(z_p)). \label{eq:ineq}\end{equation}
But we also know that $\eta_j(\arg(c \cdot z_p)) > \eta_j(\arg(c \cdot z_q))$, because 
\[ \eta_j(\arg(c \cdot z_p)) = \eta_j(\arg(z_p)) + \eta_j(\arg(c)) + \mu \]
and
\[ \eta_j(\arg(c \cdot z_q)) = \eta_j(\arg(z_q)) + \eta_j(\arg(c)) + \mu.\]
Therefore, we also have for every $\ell \notin \{p,q\}$ the inequality 
\[ \eta_j(\arg(c \cdot z_q)) < \eta_j(\arg(c \cdot z_{\ell})) < \eta_j(\arg(c \cdot z_p)),\]
which implies that  
\[ \eta_j(\arg(z_q)) + \mu < \eta_j(\arg(z_{\ell})) + \eta_j(\theta_{\ell}) < \eta_j(\arg(c \cdot z_p)) + \mu. \]
Combining this with inequality (\ref{eq:ineq}), we get that for every $\ell \in \llbracket 1,N \rrbracket$, the equality $\eta_j(\theta_{\ell}) = \mu$ holds.\\

Let us sum up the situation. If for some $j \in \llbracket 1,d \rrbracket$, we are not in the exceptional case, then we are done. But there must exist such a $j$, because otherwise we would have that 
$\eta_j(\theta_1) = \ldots = \eta_j(\theta_N)$
for every $j$, that is to say $\theta_1 = \ldots = \theta_N$.
\end{proof} 

\begin{lemma}
\label{lm:unique_proj}
Let $E$ be a compact simple subset of $\T^d$, with finitely many connected components $E_1, \ldots, E_N$. \ Assume that for every $j \in \llbracket 1, d \rrbracket$ and $p \in \llbracket 1, N \rrbracket$, there exists a unique point $z_{j,p}^-$ (respectively $z_{j,p}^+$) such that for every $b \in \T^d$ with $b \cdot E$ very simple,
\[ \min_{z \in E_p} \eta_j(\arg(b \cdot z)) = \eta_j(\arg(b \cdot z_{j,p}^-)) \]
(respectively $\max_{z \in E_p} \eta_j(\arg(b \cdot z)) = \eta_j(\arg(b \cdot z_{j,p}^+))$). If $E$ satisfies the assumptions of Lemma \ref{lm:generic}, then the set 
$F = \{ z_{1,1}^-, z_{1,1}^+, \ldots, z_{d,N}^-, z_{d,N}^+ \}$
satisfies the assumptions of Lemma \ref{lm:bad_sets}.
\end{lemma}

\begin{proof}
The result can be deduced from the following observation: for every $j \in \llbracket 1, d \rrbracket$ and $p \in \llbracket 1, N \rrbracket$, and for every $b \in \T^d$ such that $b \cdot E$ is very simple,
$\mathrm{diam}(\eta_j(\arg(b \cdot E))) = \mathrm{diam}(\eta_j(\arg(b \cdot F))).$
\end{proof}

\begin{proof}[Proof of Lemma \ref{lm:generic}]
Firstly, we can slightly modify the connected components of $E$ in order to get an $\varepsilon$-close set $\tilde{E}_{\varepsilon}$ satisfying the assumption in the previous lemma (see Figure \ref{fig:lemma24}), because if this assumption is true for one $b$ such that $b \cdot E$ is very simple, it is true for all such $b$. To this new set $\tilde{E}_{\varepsilon}$, we then associate a set 
$F_{\varepsilon} = \{  z_{1,1}^-(\varepsilon), z_{1,1}^+(\varepsilon), \ldots, z_{d,N}^-(\varepsilon), z_{d,N}^+(\varepsilon)\}$
as in this lemma. Then equalities as in Lemma \ref{lm:bad_sets} occur for $F_{\varepsilon}$ for a certain number of couples $(b,c)$ of admissible points. But recall that there is only a finite number of different values of 
$(\theta_1^{(b,c)}, \ldots, \theta_N^{(b,c)})$
that we can obtain by changing $(b,c)$. Hence, given $\varepsilon > 0$ small enough, by performing small perturbations of the connected components of $\tilde{E}_{\varepsilon}$ around the points 
$z_{i,j}^{\pm}(\varepsilon)$ of $F_{\varepsilon}$, we can construct a set $E_{\varepsilon}$ which is $\varepsilon$-close to $\tilde{E}_{\varepsilon}$ with respect to the Hausdorff distance and such that no equality as in Lemma \ref{lm:bad_sets} ever occurs, which means that $E_{\varepsilon}$ is generic.
\end{proof}

\begin{figure}[h]
\includegraphics{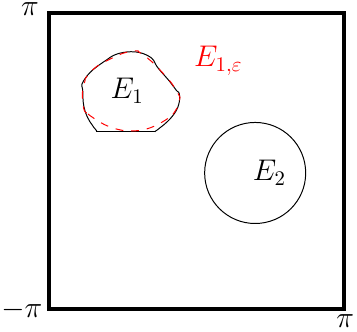}
\caption{Approximating $E$ by a set satisfying the assumptions of Lemma \ref{lm:unique_proj}.}
\label{fig:lemma24}
\end{figure}

Before concluding this section, we note that the convex hull thus constructed is compatible with rotations, 
in the sense that if $E \subset \T^d$ is as in Definition \ref{def:CH_simple}, then
$\mathrm{Convex \ Hull}_{\mathbb{T}^d}(b \cdot E) = b \cdot \mathrm{Convex \ Hull}_{\mathbb{T}^d}(E)$
for every $b \in \T^d$. Indeed, if $c \in \T^d$ is admissible for $c \cdot E$, then $c \cdot b^{-1}$ is admissible for $b \cdot E$, and, if $C = \mathrm{Convex \ Hull}_{\mathbb{T}^d}(b \cdot E)$, then 
\[ C  = (c \cdot b^{-1})^{-1} \cdot \exp \left(i \ \mathrm{Convex \ Hull}(\arg(c \cdot b^{-1} \cdot b \cdot E)) \right),\]
which yields
\[ C = b \cdot c^{-1} \cdot \exp \left(i \ \mathrm{Convex \ Hull}(\arg(c  \cdot E)) \right) = b \cdot \mathrm{Convex \ Hull}_{\mathbb{T}^d}(b \cdot E).\]

\begin{remark}
We will not give a definition of the convex hull of a general subset of the torus, since we will always keep these assumptions of compactness and finite number of connected components. However, our definition allows us to handle, in particular, compact connected subsets and sets consisting of a finite number of points. Back to our initial problem, the former corresponds to the closed image of a joint principal symbol, while the latter corresponds to the joint spectrum of a family of pairwise commuting operators acting on finite-dimensional spaces. Moreover, computing the convex hull of a finite number of points on tori seems to be of interest in computational geometry \cite{GriMar}. 
\end{remark}

\subsection*{Convex hull for compact, connected subsets of $\T^d$}

We turn to the definition of the convex hull for compact connected non necessarily simple subsets.
\begin{lemma}
\label{lm:cv_simple}
Let $E$ be a compact connected subset of $\mathbb{T}^d$. Assume that there exists a sequence $(E_n)_{n \geq 1}$ of compact connected very simple subsets such that 
\begin{enumerate} 
\item $E_n \underset{n \to \infty}{\longrightarrow} E$ with respect to the Hausdorff distance,
\item $E_{n} \subset E_{n+1}$,
\item $d_H^{\T^d}(E_n,E_{n+1}) \leq  \frac{1}{2^n} \min  \left(1, d\left(\arg(E_n),\partial\left([-\pi,\pi]^d \right) \right) \right)$.
\end{enumerate}
We call such a sequence a \emph{very simple approximation of $E$}. Then there exists a compact subset $C \subset \mathbb{T}^d$ such that the sequence $(C_n)_{n \geq 1}$ of subsets of $\T^d$ defined by 
$C_n = \mathrm{Convex \ Hull}_{\mathbb{T}^d}(E_n)$
converges to $C$ for the Hausdorff distance topology.
\end{lemma}

Before proving this result, we state the following useful lemma. It is a standard exercise to show that in $\mathbb{R}^d$, taking the convex hull is a $1$-Lipschitz operation for the Hausdorff distance; it turns out that the same does not hold in general for simple subsets of $\mathbb{T}^d$, see Figure \ref{fig:lipschitz} for a counterexample. However, the following weaker version of this property holds.

\begin{lemma}
\label{lm:lips_CH}
Let $E, F \subset \T^d$ be compact connected very simple subsets such that 
$d_H^{\T^d}(E,F) \leq  \frac{1}{2} d\left(\arg(F),\partial\left([-\pi,\pi]^d \right) \right).$
Then 
\[ d_H^{\T^d}\left(\mathrm{Convex\ Hull}_{\T^d}(E), \mathrm{Convex\ Hull}_{\T^d}(F)\right) \leq d_H^{\T^d}(E,F).\]
\end{lemma}

\begin{figure}[h]
\includegraphics{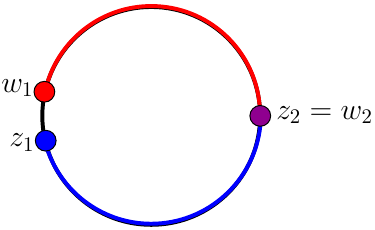}
\caption{Two very simple subsets $E = \{z_1,z_2\}$ and $F = \{w_1,w_2\}$ of $\mathbb{S}^1$ whose convex hulls (in blue, the convex hull of $E$, in red, the convex hull of $F$) are at Hausdorff distance greater than the Hausdorff distance between $E$ and $F$.}
\label{fig:lipschitz}
\end{figure}

\begin{proof}
Before starting the proof, we recall that, because of Definition \ref{def:CH_simple} and the remark following it, 
$\mathrm{Convex \ Hull}_{\T^d}(E) = \exp\left( i \left( \mathrm{Convex \ Hull}_{\T^d}(\arg(E)) \right) \right)$
and similarly for $F$.\\

Let $z \in \mathrm{Convex \ Hull}_{\T^d}(F)$; there exists $\theta \in \mathrm{Convex \ Hull}(\arg(F))$ such that $z = \exp(i \theta)$. Thus $\theta$ can be written as a finite linear combination of elements of $\arg(F)$: 
there exist $\alpha_1, \ldots, \alpha_m \in [0,1]$ and $\theta^{1}, \ldots, \theta^d \in \arg(F)$ such that 
$\sum_{\ell = 1}^m \alpha_{\ell} = 1$ and $\theta = \sum_{\ell = 1}^m \alpha_{\ell} \theta^{\ell}$
(here we use superscripts to avoid confusion with the components of elements of $\mathbb{R}^d$). For $1 \leq \ell \leq m$, let $z^{\ell} = \exp(i \theta^{\ell}) \in F.$
Fix $1 < \gamma < 2$; there exists $w^1, \ldots, w^m \in E$ such that for $1 \leq \ell \leq m$
\begin{equation} d^{\T^d}(z^{\ell}, w^{\ell}) \leq \gamma \ d_H^{\T^d}(E,F) \leq \gamma \delta, \label{eq:dist_lm_simple}\end{equation}
where $\delta = \frac{1}{2} d\left(\arg(F),\partial\left([-\pi,\pi]^d \right) \right)$. Consider 
\[ w = \exp\left(i \sum_{\ell = 1}^m \alpha_{\ell} \arg(w^{\ell})\right) \in \mathrm{Convex \ Hull}_{\T^d}(E).\]
For $\ell \in \llbracket 1, m \rrbracket$, choose a non-zero $\phi^{\ell} \in (2\pi \Z)^d$; then $\arg(w^{\ell}) - \phi^{\ell}$ does not belong to $[-\pi,\pi]^d$, thus 
\[ \left\| \arg(z^{\ell}) - \arg(w^{\ell}) + \phi^{\ell} \right\|_{\R^d} \geq d\left(\arg(F),\partial\left([-\pi,\pi]^d \right) \right),  \]
which implies, by the choice of $\gamma$, that
\[ \left\| \arg(z^{\ell}) - \arg(w^{\ell}) + \phi^{\ell} \right\|_{\R^d} \geq \gamma \delta \geq\left\| \arg(z^{\ell}) - \arg(w^{\ell})  \right\|_{\R^d}. \]
Therefore  
$d^{\T^d}(z^{\ell}, w^{\ell}) = \| \theta^{\ell} - \arg(w^{\ell})\|_{\R^d}.$
Combining this equality with the fact that 
\[ d^{\T^d}(z, w) \leq \left\| \sum_{\ell = 1}^m \alpha_{\ell} \left(\arg(w^{\ell}) - \theta^{\ell} \right)\right\|_{\R^d} \leq \sum_{\ell = 1}^m  \alpha_{\ell} \left\| \arg(w^{\ell}) - \theta^{\ell} \right\|_{\R^d}, \]
and equation (\ref{eq:dist_lm_simple}), 
\[ d^{\T^d}(z, w) \leq \gamma \sum_{\ell = 1}^m  \alpha_{\ell} \ d_H^{\T^d}(E,F) = \gamma d_H^{\T^d}(E,F).\]
Hence $d^{\T^d}(z,\mathrm{Convex \ Hull}_{\T^d}(E)) \leq \gamma d_H^{\T^d}(E,F).$ 

Exchanging  roles of $E$, $F$, we have that for every $w \in \mathrm{Convex \ Hull}_{\T^d}(E)$, the inequality 
$d^{\T^d}(w,\mathrm{Convex \ Hull}_{\T^d}(F)) \leq \gamma d_H^{\T^d}(E,F)$
holds. Thus, from the following characterization of the Hausdorff distance (see for instance \cite[Exercise $7.3.2$]{Bur}):
\begin{equation} d_H^{\mathbb{T}^d}(A,B) \leq r \Leftrightarrow \left( \forall a \in A, \ d^{\T^d}(a,B) \leq r \ \mathrm{and} \ \forall b \in B, \ d^{\T^d}(b,A) \leq r \right), \label{eq:char_haus}\end{equation}
we deduce that 
$d_H^{\mathbb{T}^d}(\mathrm{Convex \ Hull}_{\T^d}(E),\mathrm{Convex \ Hull}_{\T^d}(F)) \leq \gamma d_H^{\T^d}(E,F).$
Since $\gamma > 1$ was arbitrary, this concludes the proof. 
\end{proof}

\begin{proof}[Proof of Lemma \ref{lm:cv_simple}]
Since $(E_n)_{n \geq 1}$ converges, it is a Cauchy sequence, and furthermore, thanks to the previous lemma, we have that for every $n \geq 1$, $d_H^{\T^d}(C_n,C_{n+1}) \leq d_H^{\T^d}(E_n,E_{n+1}).$ 
Therefore, the triangle inequality yields, for every $n,p \geq 1$,
$d_H^{\T^d}(C_n,C_{n+p}) \leq \sum_{\ell = n}^{n + p-1} d_H^{\T^d}(E_{\ell},E_{\ell+1}).$
But the series 
$\sum_{\ell \geq 1} d_H^{\T^d}(E_{\ell},E_{\ell+1})$
converges; this implies that the sequence $(C_n)_{n \geq 1}$ is a Cauchy sequence as well. But it is a well-known fact that the set of compact subsets of a complete metric space, endowed with the Hausdorff distance, is complete \cite[Proposition $7.3.7$]{Bur}. Applying this to our context, we get that the sequence $(C_n)_{n \geq 1}$ converges to some compact subset $C \subset \mathbb{T}^d$.
\end{proof}

The limit obtained in Lemma \ref{lm:cv_simple} is in fact unique, in the sense that it only depends on $E$ and not on the very simple sets used to approximate it.

\begin{lemma}
If $E$ is a compact connected subset of $\T^d$ and $(E_n)_{n \geq 1}$, $(F_n)_{n \geq 1}$ are two very simple approximations of $E$ (see Lemma \ref{lm:cv_simple} for terminology), then
$\lim_{n \to +\infty} \mathrm{Convex \ Hull}_{\mathbb{T}^d}(E_n) = \lim_{n \to +\infty} \mathrm{Convex \ Hull}_{\mathbb{T}^d}(F_n),$
where, as usual, the limit is taken with respect to the Hausdorff distance.
\end{lemma}

\begin{proof}
For $n \geq 1$, let
$C_n = \mathrm{Convex \ Hull}_{\mathbb{T}^d}(E_n)$ and $D_n = \mathrm{Convex \ Hull}_{\mathbb{T}^d}(F_n).$
Denote by $C$ (respectively $D$) the limit of the sequence $(C_n)_{n \geq 1}$ (respectively $(D_n)_{n \geq 1}$). By the triangle inequality, we have that 
\begin{equation} d_H^{\mathbb{T}^d}(C,D) \leq d_H^{\mathbb{T}^d}(C,C_n) + d_H^{\mathbb{T}^d}(C_n,D_n) + d_H^{\mathbb{T}^d}(D_n,D). \label{eq:triangle}\end{equation}
The first and third terms on the right hand side of this inequality go to zero as $n$ goes to infinity. Let us estimate the second one. For every $p \geq n$, we have, using the triangle inequality again:
\[ \begin{split} d_H^{\T^d}(E_n,F_n) & \leq \frac{1}{2^n} \Big( \min \left(1, d\left(\arg(E_n),\partial\left([-\pi,\pi]^d \right) \right) \right) \\
& + \min \left(1, d\left(\arg(F_n),\partial\left([-\pi,\pi]^d \right) \right) \right)  \Big) + d_H^{\T^d}(E_{n+1},F_{n+1}). \end{split}  \]
By iterating and using the fact that 
\[ d\left(\arg(E_{n+1}),\partial\left([-\pi,\pi]^d \right) \right) \leq d\left(\arg(E_n),\partial\left([-\pi,\pi]^d \right) \right) \]
since $E_n \subset E_{n+1}$, we obtain that for every $p \geq 1$
\[  \begin{split} d_H^{\T^d}(E_n,F_n) & \leq \frac{1}{2^n} \left( \sum_{\ell = 0}^{p} \frac{1}{2^{\ell}} \right) \Big( \min \left(1, d\left(\arg(E_n),\partial\left([-\pi,\pi]^d \right) \right) \right) \\
& + \min \left(1, d\left(\arg(F_n),\partial\left([-\pi,\pi]^d \right) \right) \right)  \Big) + d_H^{\T^d}(E_{n+p},F_{n+p}). \end{split}   \]
Letting $p$ go to infinity, we deduce that 
\[ d_H^{\T^d}(E_n,F_n) \leq \frac{1}{2^{n-1}} \max \left(d\left(\arg(E_n),\partial\left([-\pi,\pi]^d \right) \right), d\left(\arg(F_n),\partial\left([-\pi,\pi]^d \right) \right) \right).  \]
Therefore, thanks to Lemma \ref{lm:lips_CH}, the second term in the right hand side of equation (\ref{eq:triangle}) satisfies
$d_H^{\mathbb{T}^d}(C_n,D_n) \leq d_H^{\mathbb{T}^d}(E_n,F_n)$
and converges to zero as well. Hence, letting $n$ go to infinity in equation (\ref{eq:triangle}), we obtain that $d_H^{\mathbb{T}^d}(C,D) = 0$, which yields $C = D$ because $C$ and $D$ are closed.
\end{proof}

Now, let $E$ be any compact connected subset of $\T^d$, and let $\mathrm{App}(E) \subset \T^d$ be the set of $a \in \T^d$ such that $a \cdot E$ admits a very simple approximation (see Lemma \ref{lm:cv_simple}). We define a map $C_E$ from $\mathrm{App}(E)$ to the set of compact subsets of $\T^d$ as follows: for $a \in \mathrm{App}(E)$, $C_E(a)$ is the limit of the convex hull of any very simple approximation of $a \cdot E$. 

\begin{definition}
\label{def:hullizable}
We say that $E$ is \emph{hullizable} if $\mathrm{App}(E) \neq \emptyset$ and the map 
$a \in \mathrm{App}(E)  \mapsto a^{-1} \cdot C_E(a)$ is constant.
\end{definition}

Figure \ref{fig:non_hullizable} displays an example of non-hullizable set.

\begin{figure}[h]
\subfigure[$C_E(a) \neq \T^2$.]{\includegraphics{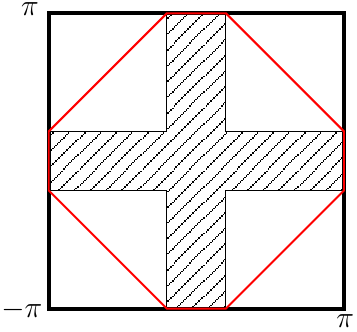}}
\hspace{3mm}
\subfigure[$C_E(a) = \T^2$.]{\includegraphics{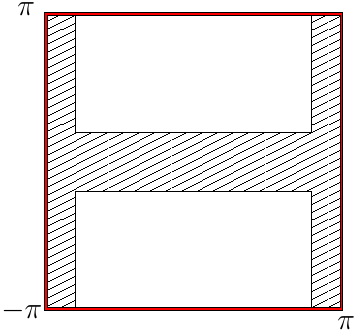}}
\caption{An example of non hullizable set $E \subset \T^2$. The figure displays $\arg(a \cdot E)$ for two different values of $a \in \T^2$, and the corresponding set $\arg(C_E(a))$ (the boundary of which is represented by a red line).}
\label{fig:non_hullizable}
\end{figure}

\begin{definition} 
Let $E$ be a hullizable compact connected subset of $\T^d$. We define the convex hull of $E$ as 
$\mathrm{Convex \ Hull}_{\T^d}(E) := a^{-1} \cdot C_E(a)$
for any $a \in \mathrm{App}(E)$.
\end{definition}

This definition agrees with Definition \ref{def:CH_simple} when $E$ is simple. Indeed, if $a \in \T^d$ is such that $a \cdot E$ is very simple, then $(F_n = a \cdot E)_{n \geq 1}$ is a very simple approximation of $a \cdot E$, hence 
\[ C_E(a) =  \mathrm{Convex \ Hull}_{\T^d}(a \cdot E) = \exp \left(i \ \mathrm{Convex \ Hull}(\arg(a \cdot E)) \right) \]
Consequently, it follows from Lemma \ref{lm:admissible} (with $N=1$) that $E$ is hullizable, and its convex hull is computed as in Definition \ref{def:CH_simple}.

\subsection*{Acknowledgements.} AP is supported by NSF grants DMS-1055897 and DMS-1518420. He also received support from ICMAT Severo Ochoa Program. YLF was supported by the European Research Council advanced grant 338809; moreover, this work was initiated while he was visiting ICMAT, Madrid, in March 2015, and he is grateful for the hospitality of the institute.

\bibliographystyle{abbrv}
\bibliography{LFPel_unitary}

\begin{thebibliography}{10}

\bibitem{Ben}
C.~M. Bender.
\newblock Introduction to {$\mathcal{PT}$}-symmetric quantum theory.
\newblock {\em Contemporary Physics}, 46(4):277--292, 2005.

\bibitem{Ber}
F.~A. Berezin.
\newblock General concept of quantization.
\newblock {\em Comm. Math. Phys.}, 40:153--174, 1975.

\bibitem{BordMeinSchli}
M.~Bordemann, E.~Meinrenken, and M.~Schlichenmaier.
\newblock Toeplitz quantization of {K}\"ahler manifolds and {${\rm gl}(N)$},
  {$N\to\infty$} limits.
\newblock {\em Comm. Math. Phys.}, 165(2):281--296, 1994.

\bibitem{BorUri}
D.~Borthwick and A.~Uribe.
\newblock Almost complex structures and geometric quantization.
\newblock {\em Math. Res. Lett.}, 3(6):845--861, 1996.

\bibitem{BouGui}
L.~Boutet~de Monvel and V.~Guillemin.
\newblock {\em The spectral theory of {T}oeplitz operators}, volume~99 of {\em
  Annals of Mathematics Studies}.
\newblock Princeton University Press, Princeton, NJ, 1981.

\bibitem{Bur}
D.~Burago, Y.~Burago, and S.~Ivanov.
\newblock {\em A course in metric geometry}, volume~33 of {\em Graduate Studies
  in Mathematics}.
\newblock American Mathematical Society, Providence, RI, 2001.

\bibitem{Cha}
L.~Charles.
\newblock Berezin-{T}oeplitz operators, a semi-classical approach.
\newblock {\em Comm. Math. Phys.}, 239(1-2):1--28, 2003.

\bibitem{Cha_BS}
L.~Charles.
\newblock Quasimodes and {B}ohr-{S}ommerfeld conditions for the {T}oeplitz
  operators.
\newblock {\em Comm. Partial Differential Equations}, 28(9-10):1527--1566,
  2003.

\bibitem{Cha_symp}
L.~Charles.
\newblock Quantization of compact symplectic manifolds.
\newblock {\em The Journal of Geometric Analysis}, 26(4):2664--2710, 2016.

\bibitem{ChaMar}
L.~Charles and J.~March\'e.
\newblock Knot state asymptotics {I}: {AJ} conjecture and {A}belian
  representations.
\newblock {\em Publications math\'ematiques de l'IH\'ES}, 121(1):279--322,
  2015.

\bibitem{ChaPelVu}
L.~Charles, A.~Pelayo, and S.~V\~u Ng{\d{o}}c.
\newblock Isospectrality for quantum toric integrable systems.
\newblock {\em Ann. Sci. \'Ec. Norm. Sup\'er. (4)}, 46(5):815--849, 2013.

\bibitem{Col}
Y.~Colin~de Verdi{\`e}re.
\newblock A semi-classical inverse problem {II}: reconstruction of the
  potential.
\newblock In {\em Geometric aspects of analysis and mechanics}, volume 292 of
  {\em Progr. Math.}, pages 97--119. Birkh\"auser/Springer, New York, 2011.

\bibitem{ColGui}
Y.~Colin~de Verdi{\`e}re and V.~Guillemin.
\newblock A semi-classical inverse problem {I}: {T}aylor expansions.
\newblock In {\em Geometric aspects of analysis and mechanics}, volume 292 of
  {\em Progr. Math.}, pages 81--95. Birkh\"auser/Springer, New York, 2011.

\bibitem{Conway}
J.~B. Conway.
\newblock {\em A course in functional analysis}, volume~96 of {\em Graduate
  Texts in Mathematics}.
\newblock Springer-Verlag, New York, 1985.

\bibitem{Dav}
E.~B. Davies.
\newblock Non-self-adjoint differential operators.
\newblock {\em Bull. London Math. Soc.}, 34(5):513--532, 2002.

\bibitem{DimSjo}
M.~Dimassi and J.~Sj{\"o}strand.
\newblock {\em Spectral asymptotics in the semi-classical limit}, volume 268 of
  {\em London Mathematical Society Lecture Note Series}.
\newblock Cambridge University Press, Cambridge, 1999.

\bibitem{DunSch}
N.~Dunford and J.~T. Schwartz.
\newblock {\em Linear operators. {P}art {III}}.
\newblock Wiley Classics Library. John Wiley \& Sons, Inc., New York, 1988.
\newblock Spectral operators, With the assistance of William G. Bade and Robert
  G. Bartle, Reprint of the 1971 original, A Wiley-Interscience Publication.

\bibitem{MR3506787}
L.~Godinho, A.~Pelayo, and S.~Sabatini.
\newblock Fermat and the number of fixed points of periodic flows.
\newblock {\em Commun. Number Theory Phys.}, 9(4):643--687, 2015.

\bibitem{GriMar}
C.~I. Grima and A.~M\'arquez.
\newblock {\em {Computational geometry on surfaces. Performing computational
  geometry on the cylinder, the sphere, the torus, and the cone.}}
\newblock {Dordrecht: Kluwer Academic Publishers. xv}, 2001.

\bibitem{HelRob}
B.~Helffer and D.~Robert.
\newblock Puits de potentiel g\'en\'eralis\'es et asymptotique semi-classique.
\newblock {\em Ann. Inst. H. Poincar\'e Phys. Th\'eor.}, 41(3):291--331, 1984.

\bibitem{HislopSigal}
P.~D. Hislop and I.~M. Sigal.
\newblock {\em Introduction to spectral theory}, volume 113 of {\em Applied
  Mathematical Sciences}.
\newblock Springer-Verlag, New York, 1996.
\newblock With applications to Schr{\"o}dinger operators.

\bibitem{HitSjo}
M.~Hitrik and J.~Sj\"ostrand.
\newblock Non-selfadjoint perturbations of selfadjoint operators in 2
  dimensions. {I}.
\newblock {\em Ann. Henri Poincar\'e}, 5(1):1--73, 2004.

\bibitem{Kos}
B.~Kostant.
\newblock Quantization and unitary representations.
\newblock {\em Uspehi Mat. Nauk}, 28(1(169)):163--225, 1973.
\newblock Translated from the English (Lectures in Modern Analysis and
  Applications, III, pp. 87--208, Lecture Notes in Math., Vol. 170, Springer,
  Berlin, 1970) by A. A. Kirillov.

\bibitem{LFPelVu}
Y.~Le~Floch, {\'A}.~Pelayo, and S.~V{\~u}~Ng{\d{o}}c.
\newblock Inverse spectral theory for semiclassical {J}aynes-{C}ummings
  systems.
\newblock {\em Math. Ann.}, 364(3-4):1393--1413, 2016.

\bibitem{LinPel}
Y.~Lin and {\'A}.~Pelayo.
\newblock Log-concavity and symplectic flows.
\newblock {\em Math. Res. Lett.}, 22(2):501--527, 2015.

\bibitem{MaMa}
X.~Ma and G.~Marinescu.
\newblock Toeplitz operators on symplectic manifolds.
\newblock {\em J. Geom. Anal.}, 18(2):565--611, 2008.

\bibitem{McDuff}
D.~McDuff.
\newblock The moment map for circle actions on symplectic manifolds.
\newblock {\em J. Geom. Phys.}, 5(2):149--160, 1988.

\bibitem{Mum}
D.~Mumford.
\newblock {\em Tata lectures on theta. {I}}.
\newblock Modern Birkh\"auser Classics. Birkh\"auser Boston, Inc., Boston, MA,
  2007.
\newblock With the collaboration of C. Musili, M. Nori, E. Previato and M.
  Stillman, Reprint of the 1983 edition.

\bibitem{MR3129048}
A.~Pelayo.
\newblock Symplectic spectral geometry of semiclassical operators.
\newblock {\em Bull. Belg. Math. Soc. Simon Stevin}, 20(3):405--415, 2013.

\bibitem{PelBAMS17}
A.~Pelayo.
\newblock Hamiltonian and symplectic symmetries: an introduction.
\newblock {\em Bull. Amer. Math. Soc. (N.S.)}, 54(3):383--436, 2017.

\bibitem{PelPolVu}
{\'A}.~Pelayo, L.~Polterovich, and S.~V{\~u}~Ng{\d{o}}c.
\newblock Semiclassical quantization and spectral limits of
  {$\hslash$}-pseudodifferential and {B}erezin-{T}oeplitz operators.
\newblock {\em Proc. Lond. Math. Soc. (3)}, 109(3):676--696, 2014.

\bibitem{PelRat}
{\'A}.~Pelayo and T.~S. Ratiu.
\newblock Circle-valued momentum maps for symplectic periodic flows.
\newblock {\em Enseign. Math. (2)}, 58(1-2):205--219, 2012.

\bibitem{MR2801777}
A.~Pelayo and S.~V\~u Ngoc.
\newblock Symplectic theory of completely integrable {H}amiltonian systems.
\newblock {\em Bull. Amer. Math. Soc. (N.S.)}, 48(3):409--455, 2011.

\bibitem{PelVuFF}
{\'A}.~Pelayo and S.~V{\~u}~Ng{\d{o}}c.
\newblock Semiclassical inverse spectral theory for singularities of
  focus-focus type.
\newblock {\em Comm. Math. Phys.}, 329(2):809--820, 2014.

\bibitem{PelVu}
{\'A}.~Pelayo and S.~V{\~u}~Ng{\d{o}}c.
\newblock Spectral limits of semiclassical commuting self-adjoint operators.
\newblock In {\em A mathematical tribute to {P}rofessor {J}os\'e {M}ar\'\i a
  {M}ontesinos {A}milibia}, pages 527--546. Dep. Geom. Topol. Fac. Cien. Mat.
  UCM, Madrid, 2016.

\bibitem{ReedSimon}
M.~Reed and B.~Simon.
\newblock {\em Methods of modern mathematical physics. {I}}.
\newblock Academic Press, Inc. [Harcourt Brace Jovanovich, Publishers], New
  York, second edition, 1980.
\newblock Functional analysis.

\bibitem{Rou}
O.~Rouby.
\newblock {Bohr-Sommerfeld Quantization Conditions for Non-Selfadjoint
  Perturbations of Selfadjoint Operators in Dimension One}.
\newblock {\em International Mathematics Research Notices}, 2017.

\bibitem{Rudin}
W.~Rudin.
\newblock {\em Functional analysis}.
\newblock International Series in Pure and Applied Mathematics. McGraw-Hill,
  Inc., New York, second edition, 1991.

\bibitem{Schli}
M.~Schlichenmaier.
\newblock Berezin-{T}oeplitz quantization for compact {K}\"ahler manifolds. {A}
  review of results.
\newblock {\em Adv. Math. Phys.}, pages Art. ID 927280, 38, 2010.

\bibitem{Schmu}
K.~Schm{\"u}dgen.
\newblock {\em Unbounded self-adjoint operators on {H}ilbert space}, volume 265
  of {\em Graduate Texts in Mathematics}.
\newblock Springer, Dordrecht, 2012.

\bibitem{Sou}
J.-M. Souriau.
\newblock Quantification g\'eom\'etrique.
\newblock {\em Comm. Math. Phys.}, 1:374--398, 1966.

\bibitem{MR3735631}
S.~Tolman.
\newblock Non-{H}amiltonian actions with isolated fixed points.
\newblock {\em Invent. Math.}, 210(3):877--910, 2017.

\bibitem{TolWei}
S.~Tolman and J.~Weitsman.
\newblock On semifree symplectic circle actions with isolated fixed points.
\newblock {\em Topology}, 39(2):299--309, 2000.

\bibitem{SanInverse}
S.~V{\~u}~Ng{\d{o}}c.
\newblock Symplectic inverse spectral theory for pseudodifferential operators.
\newblock In {\em Geometric aspects of analysis and mechanics}, volume 292 of
  {\em Progr. Math.}, pages 353--372. Birkh\"auser/Springer, New York, 2011.

\bibitem{Zel}
S.~Zelditch.
\newblock Inverse spectral problem for analytic domains. {II}. {$\Bbb
  Z_2$}-symmetric domains.
\newblock {\em Ann. of Math. (2)}, 170(1):205--269, 2009.

\end{thebibliography}

\noindent\\
Yohann Le Floch\\
Institut de Recherche Math\'ematique Avanc\'ee\\
UMR 7501, Universit\'e de Strasbourg et CNRS\\
7 rue Ren\'e Descartes\\
67000 Strasbourg, France\\
{\em E\--mail}: \texttt{ylefloch@unistra.fr}\\
\noindent
\\
{\'A}lvaro Pelayo \\
University of California, San Diego\\
Department of Mathematics \\
9500 Gilman Dr \#0112\\
La Jolla, CA 92093-0112, USA.\\
{\em E\--mail}: \texttt{alpelayo@math.ucsd.edu} 

\end{document}